\documentclass[12pt]{amsart}
\usepackage{geometry}  
\usepackage{graphicx}
\usepackage{amssymb}
\usepackage{amsmath}
\usepackage{color}
\usepackage{hyperref}
\usepackage{pdfsync}
 \usepackage{tikz}
 \usepackage{tikz-cd}

\usepackage[all]{xy}
\xyoption{matrix}
\xyoption{arrow}

\usetikzlibrary{arrows,decorations.pathmorphing,backgrounds,positioning,fit,petri}

\newtheorem{thm}{Theorem}[section]
\newtheorem{lem}[thm]{Lemma}
\newtheorem{cor}[thm]{Corollary}
\newtheorem{prop}[thm]{Proposition}

\theoremstyle{definition}
\newtheorem{defn}[thm]{Definition}
\newtheorem{eg}[thm]{Example}

\theoremstyle{remark}
\newtheorem{rem}[thm]{Remark}
 
\numberwithin{equation}{section}

\newcommand{\vs}[1]{\vskip .#1 cm}

 \newcommand{\onto}{\twoheadrightarrow}

\newcommand{\ot}{\leftarrow}

\DeclareMathOperator{\Hom}{Hom}%
\DeclareMathOperator{\Ext}{Ext}%
\DeclareMathOperator{\undim}{\underline{dim}}
 
\newcommand{\field}[1]{\mathbb{#1}}
\newcommand{\ZZ}{\ensuremath{{\field{Z}}}}
\newcommand{\CC}{\ensuremath{{\field{C}}}}
\newcommand{\RR}{\ensuremath{{\field{R}}}}

\newcommand{\QQ}{\ensuremath{{\field{Q}}}}

\newcommand{\commentout}[1]{}

\newcommand{\cC}{\ensuremath{{\mathcal{C}}}}

\newcommand{\cI}{\ensuremath{{\mathcal{I}}}}

\newcommand{\cS}{\ensuremath{{\mathcal{S}}}}

\newcommand{\cT}{\ensuremath{{\mathcal{T}}}}
\newcommand{\cU}{\ensuremath{{\mathcal{U}}}}

\newcommand{\cW}{\ensuremath{{\mathcal{W}}}}

\newcommand{\cX}{\ensuremath{{\mathcal{X}}}}
\newcommand{\cY}{\ensuremath{{\mathcal{Y}}}}

\newcommand{\no}[1]{}

\title{Short history of signed exceptional sequences}
\author{Kiyoshi Igusa}
\address{Department of Mathematics, Brandeis University, Waltham, MA 02454}\email{igusa@brandeis.edu}
\thanks{First author supported by Simons Foundation Grant \#686616}
\author{Gordana Todorov}
\address{Department of Mathematics, Northeastern University, Boston, MA 02115}\email{g.todorov@northeastern.edu}

\date{\today} 

\subjclass[2020]{
16G20; 20F55}

\keywords{picture groups, picture spaces, semi-invariants, wall-and-chamber structure, cluster category, cluster-tilting objects, Stasheff associahedron}

\begin{document}

\begin{abstract}
Whereas exceptional sequences have a long history with many well-known connections to combinatorics, signed exceptional sequences are relatively recent. The authors introduced this concept in 2017 \cite{IT13} although it was retroactively realized that the category of noncrossing partitions \cite{Cat0} is a special case of this construction. Buan and Marsh \cite{BM} have introduced the concept of $\tau$-exceptional sequences to generalize the definitions and theorems to all finite dimensional algebras. This short paper is the story of the original concept of signed exceptional sequences for hereditary algebras and how it developed out of the two authors' study of algebraic K-theory, link invariants and cluster combinatorics.
\end{abstract}

\maketitle

\tableofcontents

\section*{Introduction} Exceptional sequences originated in the theory of sequences of line bundles over weighted projective space \cite{Rudakov} and the action of the braid group on such sequences. This was almost immediately converted into analogous constructions and theorems about representations of quivers. Combinatorial versions of these constructions seem to have been known much longer \cite{Hurwitz}, \cite{Loo}. We will stick to the quiver representation formulation of these objects for now.

Let $\Lambda$ be a finite dimensional hereditary algebra over a field $K$. Often we assume $\Lambda=KQ$ the path algebra of a quiver $Q$ without oriented cycles. We call a $\Lambda$-module $E$ \emph{exceptional} if it is rigid, meaning $\Ext^1_\Lambda(E,E)=0$, or equivalently $\Hom_\Lambda(E,\tau E)=0$ and a \emph{brick} which means the endomorphism ring of $E$ is a division algebra, i.e., any nonzero endomorphism of $E$ is an automorphism of $E$. If $Q$ is a Dynkin quiver, all indecomposable $KQ$-modules are rigid bricks.

An \emph{exceptional sequence} of length $k$ is a sequence of exceptional modules $(E_1,\cdots,E_k)$ with no backward Hom or Ext, i.e.,
\[
\Hom_\Lambda(E_j,E_i)=0=\Ext^1_\Lambda(E_j,E_i)
\]
for $i<j$. The exceptional sequence is called \emph{complete} if it is maximal, i.e., $k=n$ the rank of $\Lambda$ which is the number of nonisomorphic simple $\Lambda$-module. This is also the number of vertices of the quiver $Q$.

Every exceptional sequence $(E_1,\cdots,E_k)$ produces a ``perpendicular category'' or, more precisely, a ``Hom-Ext perpendicular category'' $(E_1\oplus\cdots\oplus E_k)^{\perp_{01}}$ which is the full subcategory of all modules $M$ so that $\Hom_\Lambda(E_i,M)=0=\Ext^1_\Lambda(E_i,M)$ for all $i\le k$. Call this category $\cW_k$. Then $\cW_k$ is a finitely generated ``wide subcategory'' of $mod\text-\Lambda$ \cite{IngTh} which is equivalent to the module category of a hereditary algebra of rank $n-k$. The exceptional sequence $E_1,\cdots,E_k$ can be completed on the left by adjoining a complete exceptional sequence of $\cW_k$. We get a complete exceptional sequence
\[
    (X_1,\cdots,X_{n-k},E_1,\cdots,E_k).
\]
In this sequence $X_{n-k}$ is called \emph{relatively projective} if it is a projective object of the perpendicular category $\cW_k$. More generally, an object $E_i$ in an exceptional sequence is relatively projective if it is a projective object of the perpendicular category of the objects to its right. In particular, the last term $E_k$ is relatively projective if and only if it is projective. A \emph{signed exceptional sequence} is an exceptional sequence together with signs (+/-) on it relatively projective terms. 

In this paper, we will tell the story of how this concept evolved. However, the narative is in heuristic rather than in chronological order.

Signed exceptional sequences arose from the study of ``pictures'', ``picture groups'' and ``picture spaces''. The broad outline is as follows. In the 1970's, for his PhD thesis work, the first author developed the concept of a ``picture'' and used it to define algebraic K-theory invariants for diffeomorphisms of smooth manifolds. Later, working with Kent Orr, pictures were used to settle a conjecture about link invariants. Around the turn of the century, Fomin and Zelevinsky developed the theory of cluster algebras \cite{FZ1,FZ2,FZ4} and the second author, together with Buan, Marsh, Reineke and Reiten related this to tilting theory. They ``categorified'' cluster theory by creating the cluster category \cite{BMRRT}. This was also done in the $A_n$ case by Caldero, Chapoton and Schiffler \cite{CCS}. Fomin-Zelevinsky clusters were shown to correspond to ``cluster-tilting objects'' in the cluster category. These were extensions of ``tilting objects'' which were sets of indecomposable modules which formed the vertices of a triangulation of the $n$-simplex. By adding shifted projective objects, this triangulation was extended to a triangulation of the $n$-sphere. The first author saw the resemblance with his ``pictures'' for the Steinberg group. Jerzy Weyman recognized these pictures as variants of the semi-invariant pictures from his work with Harm Derksen. The four authors: Igusa, Orr, Todorov, Weyman combined their ideas in \cite{IOTW}, \cite{IOTW2}.

Next, came the study of picture groups and their homology. The first author had developed ``pictures'' to study the third homology of groups, the Steinberg group for his thesis and torsion-free nilpotent groups for studying link invariants with Kent Orr. The question was: what was the underlying group for the semi-invariant pictures of \cite{IOTW}, \cite{IOTW2}? The two authors of this paper, together with Jerzy Weyman, invented ``picture groups'' and computed their homology and cohomology in the case of quivers of type $A_n$ \cite{ITW16}. To do that, they computed the homology of the ``picture space'' which is a simplicial complex with simplices corresponding to cluster-tilting objects. They needed to know that the picture space is a $K(\pi,1)$ with fundamental group the picture group and therefore has the same integral homology as the picture group.

There are several proofs that the picture space is a $K(\pi,1)$. The most satisfying proof shows that the space is a cubical complex which is locally CAT(0) and, therefore, a $K(\pi,1)$. To covert a simplicial complex into a cubical complex, we take the category of simplices with inclusion maps. This is the cluster morphism category. Sequences of minimal simplex inclusions are signed exceptional sequences. Every simplex (corresponding to a cluster tilting object) produced a cube of morphisms. There are $n!$ edge paths for each cube and therefore we obtain $n!$ signed exceptional sequences for every cluster-tilting object.

Outline of this paper: In Section \ref{sec 1: pictures}, we define ``pictures'' and discuss their relation to homology of groups. Pictures are planar diagrams labeled with generators and relations of a group $G$. Pictures represent elements of $H_3G$, the third integral homology of the group $G$ (Theorem \ref{thm: pictures give H3G}). For the Steinberg group of a ring $R$, $H_3(St(R))$ is isomorphic to $K_3(R)$, the third algebraic $K$-theory group of $R$ (Theorem \ref{Gersten Thm}). In Section \ref{sec 2: cluster categories}, we discuss cluster categories and cluster-tilting objects in a cluster category. We indicate how, for finite type, this gives a triangulation of the unit sphere $S^{n-1}$ in $\RR^n$. In Section \ref{sec 3: semi-invariant pictures}, we show how the ``walls'' of the picture are given as domains of definition of semi-invariants on the representation space. In Section \ref{sec 4: picture group} we define picture groups and outline the construction of picture spaces for Dynkin quivers. One example had already been constructed by Loday. He constructed what we now call the picture space of type $A_n$ by identifying faces of the Stasheff associahedron. We go over Loday's example and generalize it. We also present our calculation of the homology of picture groups of type $A_n$ assuming the picture space is a $K(\pi,1)$ for the picture group. In Section \ref{sec 5: simplices to cubes} we explain how dualizing a simplex makes it into a cube. So, picture space is seen to be a union of cubes, one for every cluster-tilting object of the category. In Section \ref{sec 6: signed exceptional sequences in general}, we construct the cluster morphism category and show that its geometric realization is a union of cubles. Finally, in Section \ref{sec 7: signed exceptional sequences} we define signed exceptional sequences and give examples.

Although torsion classes were not used in the development of signed exceptional sequences, they are a popular area of study and we incorporate them into the pictures in the final section.

\section{Pictures and homology of groups}\label{sec 1: pictures} 

The story begins in 1979, with the first author's PhD thesis \cite{thesis} in which he introduced the concept of ``pictures'' to find an algebraic K-theory invariant for the fundamental group of certain diffeomorphism spaces of smooth manifolds. There were earlier versions of the idea of pictures. They were sometimes called ``spherical diagrams'' \cite{LS}. But Igusa was able to make detailed calculations using pictures to compute new invariants in algebraic topology \cite{GrInvRedrawn}, \cite{IgOrr}, \cite{IgKlein}. The original definition was 2-dimensional, a planar diagram.

\begin{defn}\label{def: 2d picture} Let $G$ be a group with a presentation $G=\left<\cX|\cY\right>$ where elements of $\cY$ (the relations) are cyclically reduced words in $\cX\cup \cX^{-1}$ and $\cY\cap\cY^{-1}=\emptyset$,
{\color{black}A \emph{picture} for $G$ is a planar multigraph consisting of labeled vertices and labeled edges with additional structure. Each vertex is labeled with an element of $\cY\cup\cY^{-1}$ and decorated with an asterisk $\ast$ in one chosen region between incident edges. We call this the ``base point direction'' at the vertex. In some cases we decorate a vertex with its vertex label which we place in the region of the base point direction and, in that case, the asterisk becomes redundant and we delete it. Each edge is oriented and labeled with an element of $\cX$. At each vertex we require that the labels of the edges incident to that vertex, with the label inverted when the edge points away from the vertex, read counterclockwise from the base point direction at the vertex, give a word in $\cX\cup \cX^{-1}$ which is the element of $\cY\cup\cY^{-1}$ which is the label at that vertex. }

We use the drawing convention that indicates the orientation of edges by having them curve to the right. For example, circles should be oriented clockwise.
\end{defn}

{\color{black}In some cases, such as Example \ref{eg: 3 circles} below, the edge labels will uniquely determine the vertex labels together with their base point directions in the sense that there is only one choice of these vertex structures which is consistent with the definition. By this we mean there is a unique region next to the vertex so that, reading the edge labels counterclockwise around the vertex starting in this region, we obtain an element of $\cY\cup\cY^{-1}$. That unique region is the base point direction and the element of $\cY\cup\cY^{-1}$ that we obtain is the vertex label. In that case, we often omit these uniquely determined labels. In Example \ref{eg: 3 circles}, we could have deleted the vertex labels since they are uniquely determined. However, for better understanding of the basic definitions, we have inserted the vertex labels and placed them in the base point direction of each vertex.}

Loday \cite{Loday} has also given a short account of definitions and proofs for all of the main properties of what he called ``Igusa's pictures''. 

\begin{eg}\label{eg: Z3 picture}
The following is a picture for $G=\ZZ_3=\left<x|x^3\right>$.
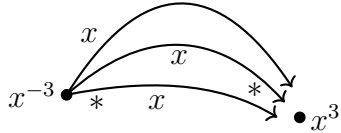
    \begin{center}
\begin{tikzpicture}

\coordinate (A) at (0,.5);
\coordinate (B) at (3.1,0.2);
\draw[fill] (A) circle[radius=2pt];
\draw[fill] (B) circle[radius=2pt];
\draw[thick,->] (A)..controls (1,2.1) and (2,2.1)..(3,.6);
\draw[thick,->] (A)..controls (1,1.4) and (2,1.4)..(2.9,.4);
\draw[thick,->] (A)..controls (1,.7) and (2,.7)..(2.8,.2);
\draw (1.2,.4) node{$x$};
\draw (1.5,1) node{$x$};
\draw (.3,1.3) node{$x$};
\draw (2.5,.57) node{$\ast$};
\draw (.4,.4) node{$\ast$};
\draw (B) node[right]{$x^3$};
\draw (A) node[left]{$x^{-3}$};
\end{tikzpicture}
\end{center}
There are three oriented edges labeled $x$. The labels of the vertices are elements of $\cY\cup\cY^{-1}$ given by reading the incoming edge labels (and inverse outgoing labels) counterclockwise starting at the base point direction indicated with $\ast$. In this picture, there are three possible base point directions at each vertex. They give the same relation ($x^3$ on the right and $x^{-3}$ on the left). However, if we were to move the base point direction at either vertex, the picture would become a different picture not isomorphic to the given picture. The orientation and curvature of edges agrees with the drawing convention that edges should be curved to the right.
\end{eg}

\begin{eg}\label{eg: 3 circles}
Here $G=\ZZ^3=\left<x,y,z|[x,y],[y,z],[z,x]\right>$. We use the ``representation theoretic'' convention that 
\[
    [x,y]:=y^{-1}xyx^{-1}
\]
instead of the usual $xyx^{-1}y^{-1}$. In \cite{GrInvRedrawn} we took our old paper and redrew the pictures using the representation theoretic convention to convert the original Morse theory pictures into semi-invariant pictures. Here is a picture for $G=\ZZ^3$ with vertices and edges labeled.
\vs2
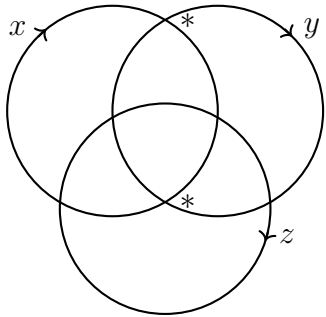
{\begin{center}
\begin{tikzpicture}[scale=1.4] 
{ 
\begin{scope}[xshift=6cm]
\draw (0.3,1.2) node{$a$};
\draw (0.14,-1.1) node[right]{$a^{-1}$};
\begin{scope}[xshift=-.7cm]
	\draw[thick] (0,0) circle[radius=1.4cm];
		\draw[thick,->] (-.985,.99)--(-.88,1.09); 
		\draw (-1,1.1) node[left]{$x$};
\end{scope}
\begin{scope}[xshift=.7cm]
	\draw[thick] (0,0) circle[radius=1.4cm];
	\draw[thick,<-] (.99,.99)--(.89,1.09); 
		\draw (1,1.1) node[right]{$y$};
		\draw (.55,-1.4) node{$b$};
		\draw (-1.55,-.45) node{$b^{-1}$};
\end{scope}
\begin{scope}[yshift=-1.3cm]
	\draw[thick] (0,0) circle[radius=1.4cm];
	\draw[thick,<-] (1.33,-.43)--(1.37,-.3); 
	\draw (1.37,-.36) node[right]{$z$};
	\draw (.45,1.55) node{$c^{-1}$};
	\draw (-1.5,0.35) node{$c$};
\end{scope}
\end{scope}
}
\begin{scope}[xshift=9cm]
\draw (0,0) node[right]{$a=[x,y]=y^{-1}xyx^{-1}$};
\draw (0,-.5) node[right]{$b=[y,z]=z^{-1}yzy^{-1}$};
\draw (0,-1) node[right]{$c=[z,x]=x^{-1}zxz^{-1}$};
\end{scope}
\end{tikzpicture}
\end{center}
}
{All six vertices are labeled with relations or inverse relations. The labels of the vertices are {placed} in the base-point direction. For example, at the lower right vertex, we have the label $b=[y,z]$. The location of the base-point direction is uniquely determined in this case since the four possible words that can be read counterclockwise at that vertex are:
\[
	yzy^{-1}z^{-1}, zy^{-1}z^{-1}y, y^{-1}z^{-1}yz, z^{-1}yzy^{-1}.
\]
Only the last one $[y,z]:=z^{-1}yzy^{-1}$ is in our list of relation or inverse relations. So, $b=[y,z]$ is the relation at that vertex and the word order dictates that we start at the location (space between edges) marked ``$b$''.}

We use the convention that smooth curves have the same label throughout. Also, when the base point directions are uniquely determined and the relations at the vertices are uniquely determined as in this example (but not in Example \ref{eg: Z3 picture}), we can omit them from the notation. In this picture, the vertex labels $a,b,c$ and their inverses can be safely omitted.
\end{eg}

\begin{thm}\label{thm: pictures give H3G}
    A picture $P$ for $G=\left<\cX|\cY\right>$ represents an element of $H_3(G)$, the third homology of $G$ if, for every $y\in\cY$, $y$ and $y^{-1}$ occur the same number of times as vertex labels. Furthermore, every element of $H_3(G)$ is represented by such a picture.
\end{thm}

The examples above satisfy this condition. So, they represent elements of $H_3(G)$. In fact these pictures give the generators of $H_3(\ZZ_3)=\ZZ_3$ and $H_3(\ZZ^3)=\ZZ$. Later, in Example \ref{fig: Heisenberg group} we will see another picture satisfying this condition. It represents the generator of $H_3(H)$, the Heisenberg group \cite{IgOrr}.
To do algebraic K-theory we start with a theorem of Gersten:

\begin{thm}\cite{Gersten}\label{Gersten Thm}
    For any ring $R$, the 3rd algebraic K-theory group $K_3(R)$ is isomorphic to $H_3(St(R))$.
\end{thm}

Here $St(R)$ represents the infinite Steinberg group which is defined as follows.

\begin{defn}\label{def: Steinberg gp}
    For any ring $R$, the \emph{Steinberg group} of $R$ is the group with generators $x_{ij}^r$ for positive integers $i\neq j$ and $r\in R$ subject to the following relations.
\begin{enumerate}
    \item[(0)] $x_{ij}^rx_{ij}^s=x_{ij}^{r+s}$
    \item $[x_{ij}^r,x_{k\ell}^s]=1$ if $j\neq k$ and $i\neq \ell$.
    \item $[x_{ij}^r,x_{jk}^s]=x_{ik}^{rs}$.
\end{enumerate}
\end{defn}
In \cite{thesis}, \cite{GrInvRedrawn} we restricted to the case of an integer group ring $R=\ZZ \pi$. In that case, we can take as generators $x_{ij}^u$ where $u\in \pi$ and take (0) not as a relation but as the definition of $x_{ij}^r$ for any $r\in \ZZ\pi$: $x_{ij}^{\sum n_ku_k}=\prod (x_{ij}^{u_k})^{n_k}$.

Combining the last two theorems we conclude that elements of $K_3(\ZZ\pi)$ were represented by certain pictures for the Steinberg group of $\ZZ\pi$. But this is not good enough since 2-parameter Morse theory produces arbitrary pictures for the Steinberg group. We needed the following.

\begin{thm}\cite{GrInvRedrawn}\label{thm: all pictures for St are good}
Every picture for $St(\ZZ\pi)$ produces an element of $K_3(\ZZ\pi)$ in a canonical way. More precisely, $K_3(\ZZ\pi)$ is isomorphic to the group of deformation classes of pictures for $St(\ZZ\pi)$ modulo the ``second order Steinberg relations''.
\end{thm}

Addition of pictures is given by disjoint union. Deformations are given by concordances of edges with the same label and cancellation of inverse relations:
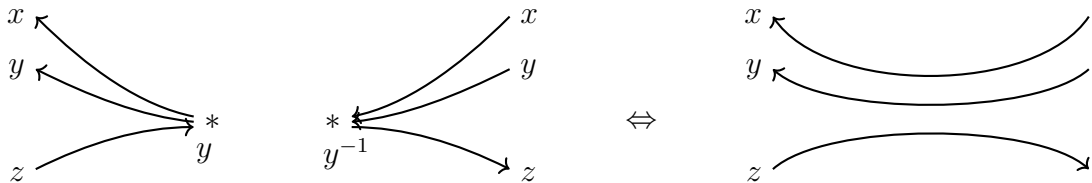
\begin{center}
\begin{tikzpicture}[scale=.7]

\coordinate (A) at (3,0);
\begin{scope}
\draw[thick,<-] (0,2)..controls (1,1) and (2,.3)..(3,0.1);
\draw[thick,<-] (0,1)..controls (1,.5) and (2,.1)..(3,0);
\draw[thick,->] (0,-.9)..controls (1,-.4) and (2,-.1)..(3,-.1);
\draw (0,2) node[left]{$x$};
\draw (0,1) node[left]{$y$};
\draw (0,-1) node[left]{$z$};
\draw (3,0) node[right]{$\ast$};
\draw(3.2,-.2) node[below]{$y$};
\end{scope}
\begin{scope}[xscale=-1,xshift=-9cm]
\draw[thick,->] (0,2)..controls (1,1) and (2,.3)..(3,0.1);
\draw[thick,->] (0,1)..controls (1,.5) and (2,.1)..(3,0);
\draw[thick,<-] (0,-.9)..controls (1,-.4) and (2,-.1)..(3,-.1);
\draw (0,2) node[right]{$x$};
\draw (0,1) node[right]{$y$};
\draw (0,-1) node[right]{$z$};
\draw (3,0) node[left]{$\ast$};
\draw (3.1,-.1) node[below]{$y^{-1}$};
\end{scope}
\draw (11.5,0) node{$\Leftrightarrow$};
\begin{scope}[xshift=14cm]
\draw[thick,<-] (0,2)..controls (1,.5) and (5,.5)..(6,2);
\draw[thick,<-] (0,1)..controls (1,.1) and (5,.1)..(6,1);
\draw[thick,->] (0,-.9)..controls (1,0) and (5,0)..(6,-.9);
\draw (0,2) node[left]{$x$};
\draw (0,1) node[left]{$y$};
\draw (0,-1) node[left]{$z$};
\end{scope}
\end{tikzpicture}
\end{center}
There are five (series of) second order relations. Two of them are shown in Figure \ref{fig: two second order relations}. We continue using the convention that arcs are oriented to curve right and base point directions are uniquely determined.
\vs2
{
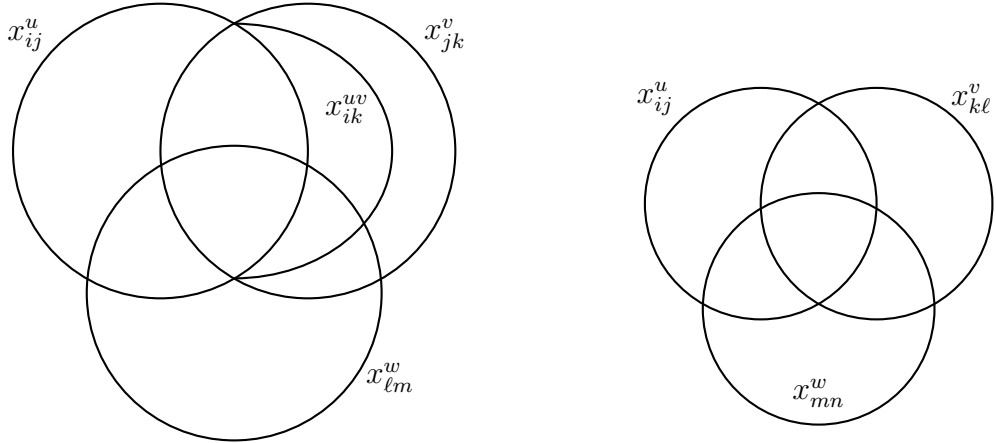
\begin{figure}[htbp]
\begin{center}
\begin{tikzpicture}[scale=1.4] 
{
\begin{scope}
\begin{scope}
\clip rectangle (2,1.3) rectangle (0,-1.3);
\draw[thick] (0,0) ellipse [x radius=1.5cm,y radius=1.21cm];
\end{scope}
\begin{scope}[xshift=.87mm]
		\draw[thick] (1.3,.4) node[left]{$x_{ik}^{uv}$}; 
\end{scope}
\begin{scope}[xshift=-.7cm]
	\draw[thick] (0,0) circle[radius=1.4cm];
		\draw (-1,1.1) node[left]{$x_{ij}^u$}; 
\end{scope}
\begin{scope}[xshift=.7cm]
	\draw[thick] (0,0) circle[radius=1.4cm];
		\draw (1,1.1) node[right]{$x_{jk}^v$}; 
\end{scope}
\begin{scope}[yshift=-1.35cm]
	\draw[thick] (0,0) circle[radius=1.4cm];
	\draw (1.15,-.8) node[right]{$x_{\ell m}^w$}; 
\end{scope}
\end{scope} 
}
{ 
\begin{scope}[xshift=5cm,yshift=-.5cm]
\draw[thick] (0,0) circle[radius=11mm];
\draw[thick] (1.1,0) circle[radius=11mm];
\draw[thick] (.55,-1) circle[radius=11mm];
\draw(-1,1) node{$x_{ij}^u$};
\draw(2,1) node{$x_{k\ell}^v$};
\draw(.55,-1.8) node{$x_{mn}^w$};
\end{scope}
}
\end{tikzpicture}
\caption{These are two ``second order Steinberg relations''. On the left we have $u,v,w\in\pi$, $i,j,k$ distinct, $\ell\neq j,k,m$ and $m\neq i,j$. We also require $x_{ik}^{uv}\neq x_{\ell m}^w$ since $[x,x]$ is not a reduced word. The figure on the right has similar constraints plus the restriction that $x_{ij}^u,x_{k\ell}^v,x_{mn}^w$ are distinct.}
\label{fig: two second order relations}
\end{center}
\end{figure}
}

{

\begin{figure}[htbp]
\begin{center}
{
\begin{tikzpicture}[scale=.6]
\coordinate (P1) at (-2,0);
\coordinate (P1p) at (-2,0.2);
\coordinate (P3) at (2,0);
\coordinate (P3p) at (2,0.2);
\coordinate (P2) at (0,3.464);
\coordinate (nP1) at (4,3.464);
\coordinate (nP3) at (-4,3.464);
\coordinate (nP2) at (0,-3.464);
\coordinate (nP1r) at (4,3.55);
\coordinate (nP3r) at (-4,3.55);
\coordinate (nP2r) at (0,-3.55);

\coordinate (S2) at (2,4);
\coordinate (S2r) at (2.2,4);
\coordinate (S3) at (1.45,-2);
\coordinate (S3r) at (1.6,-2);
\coordinate (X) at (4,0);
\coordinate (Y) at (4.5,2);

\coordinate (I2) at (3.6,1.7);
\coordinate (I2r) at (3.7,1.7);

\begin{scope}
\clip (P2) circle[radius=4cm];
\draw[thick,red] (-.85,2) circle[radius=3.464cm];
\end{scope}

\begin{scope}
\clip (P3) circle[radius=4cm];
\draw[thick] (1.15,1.464) circle[radius=3.464cm];
\end{scope}

\draw[thick] (S3)..controls (X) and (Y)..(S2);

\draw[thick,blue] (P1) circle[radius=4cm];
\draw[thick] (P2) circle[radius=4cm];
\draw[thick] (P3) circle[radius=4cm];

\draw[blue] (-6,0) node[left]{$x_{12}^u$};
\draw (6,0) node[right]{\small$x_{34}^w$};
\draw (4,6) node{\small$x_{23}^v$};
\draw (2.5,-1.9) node[right]{\small$x_{24}^{vw}$};
\draw (2.5,0) node[right]{\tiny$x_{14}^{uvw}$};
\draw[red] (-.2,5.9) node{$x_{13}^{uv}$};

\end{tikzpicture}
}
\caption{This is one more second order relation which will be used later. The colors will be explained.}
\label{Fig: A3 picture with xij}
\end{center}
\end{figure}
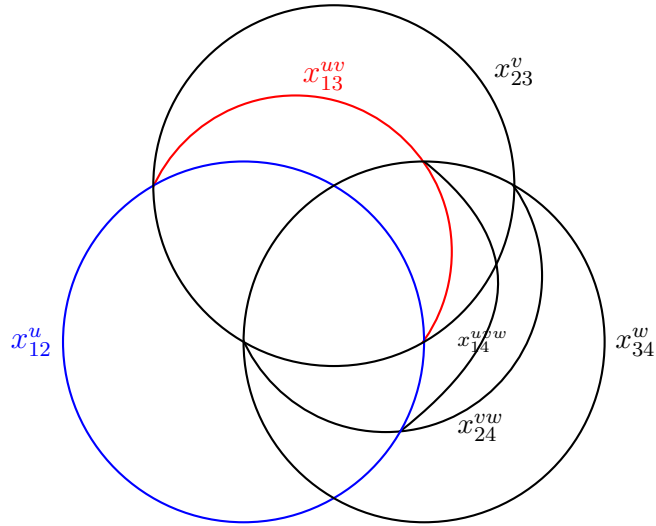

}

In 2001, Kent Orr and the first author in \cite{IgOrr} solved a conjecture about Milnor's $\overline\mu$ invariants using pictures and the homology of torsion-free nilpotent group. Such a group $N$ is a discrete subgroup of its Malcev completion which is a Euclidean space. So, the quotient is a closed manifold which is also a $K(\pi,1)$ for $\pi=N$. Igusa and Orr needed to compute the homology of these groups and they used pictures for degree 3 homology and higher dimensional pictures to compute homology in higher degrees. They used the terms ``atoms'' and ``molecules'' to describe how these pictures were made up of smaller pieces. A later work of the current two authors \cite{IT14} uses a representation theoretic version of this idea to obtain a characterization of maximal green sequences in terms of the picture monoid, i.e., they are the words in the picture monoid with product equal to the Coxeter element.

In 1993, John Klein and the first author used pictures in \cite{IgKlein} to make the first explicit computation of the higher Reidemeister torsion.

We summarize this section by saying that elements of the third homology of groups can be represented by 2-dimensional pictures and that, with this idea, pictures were used to obtain many new results in algebraic topology.

\section{Cluster categories and cluster tilting objects}\label{sec 2: cluster categories} 

{
In the year 2000, Fomin and Zelevinsky introduced the concept of ``cluster algebras'' and used this idea to solve many problems in combinatorics \cite{FZ1},\cite{FZ2},\cite{FZ4}. With four coauthors, the second author was able to ``categorify'' cluster algebras by constructing a new triangulated category which they called the ``cluster category''. This is was constructed as an orbit category of the derived category of modules over an hereditary algebra. The fundamental domain of this orbit category was the full subcategory given by $mod\text-\Lambda$ and the shifted projective modules $P_i[1]$.

Clusters in the cluster categories correspond to cluster-tilting objects in a cluster category and, for finite type, we will see that this gives a triangulation of the unit sphere $S^{n-1}$ in $\RR^n$. 
}
\vs2

{
For $\Lambda$ a hereditary algebra of rank $n$, a tilting module is a rigid module $T$ which is the direct sum of $n$ exceptional modules $T=\oplus T_i$. For $\Lambda$ of finite representation type, the dimension vectors of the exceptional modules form the vertices of a triangulation of the $(n-1)$-simplex $\Delta^{n-1}$ where the $n$ vertices of each $n-1$ dimensional simplicies of the triangulation give the tilting modules.
\begin{eg}\label{eg: A3 tilting modules}
Take $\Lambda=KQ$, the path algebra of the quiver $Q:1\ot 2\ot 3$. This has six exceptional modules $S_1,S_2,S_3, P_2,P_3,I_2$ with dimension vectors $(1,0,0), (0,1,0), (0,0,1),(1,1,0), (1,1,1), (0,1,1)$.\vs2

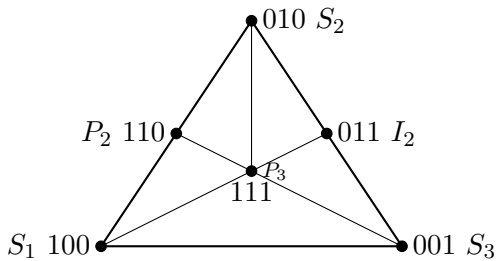
\begin{minipage}{0.4\textwidth}
\begin{center}
\begin{tikzpicture}
\coordinate (S1) at (-2,0);
\coordinate (S2) at (0,3);
\coordinate (S3) at (2,0);
\coordinate (P2) at (-1,1.5);
\coordinate (I2) at (1,1.5);
\coordinate (P3) at (0,1);
\draw (S1)--(I2);
\draw (S3)--(P2);
\draw[thick] (S2)--(S3)--(S1)--(S2);
\draw (S2)--(P3);
\draw (S1) node[left]{\small$S_1$ 100};
\draw (P2) node[left]{\small$P_2$ 110};
\draw (I2) node[right]{\small011 $I_2$};
\draw (S2) node[right]{\small010 $S_2$};
\draw (S3) node[right]{\small 001 $S_3$};
\draw (P3) node[below]{\small $111$};
\draw (P3) node[right]{\tiny$P_3$};
	\draw[fill] (S1) circle[radius=2pt];
	\draw[fill] (S2) circle[radius=2pt];
	\draw[fill] (S3) circle[radius=2pt];
	\draw[fill] (P2) circle[radius=2pt];
	\draw[fill] (I2) circle[radius=2pt];
	\draw[fill] (P3) circle[radius=2pt];
\end{tikzpicture}
\end{center}

\end{minipage}
\hfill 
\begin{minipage}{0.45\textwidth}
The vertices of the five triangles give the 5 tilting modules:

\quad $S_1\oplus P_2\oplus P_3$

\quad$S_2\oplus P_2\oplus P_3$

\quad$S_2\oplus I_2\oplus P_3$

\quad$S_3\oplus I_2\oplus P_3$

\quad$S_3\oplus S_1\oplus P_3$

\end{minipage}%

\end{eg}
} 

{
A cluster-tilting object is $T\oplus P[1]$ where $T$ is a rigid module with, say, $k$ exceptional components $T=T_1\oplus \cdots\oplus T_k$ and $P$ is a projective module with $n-k$ components so that $\Hom_\Lambda(P,T)=0$. Using that $\Ext^1_\Lambda(P[1],M)=\Hom_\Lambda(P,M)$, we see that $T\oplus P[1]$ is rigid, i.e., its components are Ext-orthogonal. For $\Lambda$ of finite representation type, the dimension vectors of the exceptional modules form the vertices of a triangulation of $\partial C(n)$, the boundary of the $n$-dimensional octahedron $C(n)$ which is the convex hull of the unit vectors in $\RR^n$ and their negative. The vertices of $C(n)$ are the $n$ unit vectors which are the dimension vectors of the simple modules $S_i$ and we put the shifted projective modules $P_i[1]$ at $-e_i$, the negative unit vectors. 

However, the boundary of $C(n)$ is triangulated with more vertices and more simplices. The $n-1$-simplices of the triangulation of $\partial C(n)$ give the cluster-tilting objects. Note that $\partial C(n)$ is homeomorphic to $S^{n-1}$, the unit sphere in $\RR^n$.

\begin{eg}\label{eg: A3 cluster tilting objects}
We continue with the path algebra of $Q:1\ot 2\ot 3$. This has 14 cluster tilting objects including the 5 tilting modules.\vs2

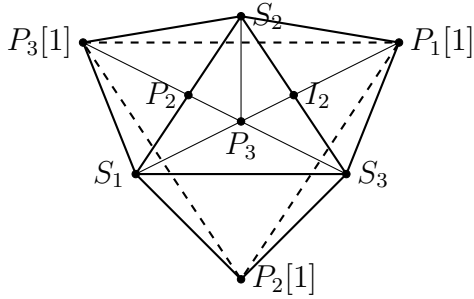
\begin{minipage}{0.4\textwidth}
\begin{center}
\begin{tikzpicture}[scale=.7]
\coordinate (S1) at (-2,0);
\coordinate (S2) at (0,3);
\coordinate (S3) at (2,0);
\coordinate (P2) at (-1,1.5);
\coordinate (I2) at (1,1.5);
\coordinate (P3) at (0,1);
\coordinate (nP1) at (3,2.5);
\coordinate (nP3) at (-3,2.5);
\coordinate (nP2) at (0,-2);
\draw (nP1)--(S1) (nP3)--(S3);
\draw[thick,dashed] (nP3)--(nP1)--(nP2)--(nP3);
\draw[thick] (nP3)--(S1)--(nP2)--(S3)--(nP1)--(S2)--(nP3);
\draw[thick] (S2)--(S3)--(S1)--(S2);
\draw (S2)--(P3);
\draw (S1) node[left]{$S_1$};
\draw (P2) node[left]{$P_2$};
\draw (I2) node[right]{$I_2$};
\draw (S2) node[right]{ $S_2$};
\draw (S3) node[right]{ $S_3$};
\draw (P3) node[below]{$P_3$};
\draw (nP1) node[right]{$P_1[1]$};
\draw (nP3) node[left]{$P_3[1]$};
\draw (nP2) node[right]{$P_2[1]$};
	\draw[fill] (S1) circle[radius=2pt];
	\draw[fill] (S2) circle[radius=2pt];
	\draw[fill] (S3) circle[radius=2pt];
	\draw[fill] (P2) circle[radius=2pt];
	\draw[fill] (P3) circle[radius=2pt];
	\draw[fill] (I2) circle[radius=2pt];
	\draw[fill] (nP3) circle[radius=2pt];
	\draw[fill] (nP2) circle[radius=2pt];
	\draw[fill] (nP1) circle[radius=2pt];
\end{tikzpicture}
\end{center}

\end{minipage}
\hfill 
\begin{minipage}{0.45\textwidth}
\,\vs2
Nine new triangles giving 9 cluster-tilting objects:\vs1

 $S_1\oplus P_2\oplus P_3[1]$\quad $P_1[1]\oplus P_2[1]\oplus P_3[1]$

$S_2\oplus P_2\oplus P_3[1]$\quad $S_1[1]\oplus P_2[1]\oplus P_3[1]$

$S_2\oplus I_2\oplus P_1[1]$\,\quad $P_1[1]\oplus S_2\oplus P_3[1]$

$S_3\oplus I_2\oplus P_1[1]$\,\quad $P_1[1]\oplus P_2[1]\oplus S_3$

$S_3\oplus S_1\oplus P_2[1]$

\end{minipage}%

\end{eg}
} 

The general result is the following.

\begin{thm}\label{thm: cluster-tilting triangulation}
For any hereditary algebra $\Lambda$ of finite representation type having $n$ simple modules, there is a triangulation of the unit sphere $S^{n-1}$ with vertices corresponding to the indecomposable objects of the cluster category and so that $n$ vertices span a top dimensional simplex in this triangulation if and only if they form a cluster tilting object, i.e., if the vertices are pairwise Ext-orthogonal.
\end{thm}

This follows from the results of \cite{BMRRT}. However, Jerzy Wayman realized that we can say more: The walls of this triangulation are domains of semi-invariants.

\section{Semi-invariant pictures}\label{sec 3: semi-invariant pictures}

{
 A semi-invariant comes from the action of a group $G$ on a vector space $X$ over a field $K$ and it is associated to a character $\chi:G\to K^\times$ in the following sense. A map $\varphi:X\to K$ is a \emph{semi-invariant with associated character $\chi$} if
\[
	\varphi(g\cdot x)=\chi(g)\varphi(x)
\] 
for all $g\in G,x\in X$.

The canonical example is $X=\Hom_K(K^n,K^n)$ and $G=GL_n(K)\times GL_n(K)$ acting on $X$ by
\[
	(g_0,g_1)(h)=g_0\circ h\circ g_1^{-1}.
\]
Then $\varphi(h)=\det(h)$ is a semi-invariant on $X$ with associated character $\chi$ given by
\[
	\chi(g_0,g_1)=\det g_0\det g_1^{-1}.
\]
Since all characters on $GL_n(K)$ are ``determinantal'', i.e. integer powers of the determinant, all characters on $G$ are given by
\[
	\chi_{ij}(g_1,g_2)=(\det g_1)^i(\det g_2)^{j}.
\]
The weight of $\chi_{ij}$ is defined to be $(i,j)\in\ZZ^2$. So, the weight of $\det:\Hom_K(K^n,K^n)\to K$ is $(1,-1)$. (See \cite[Appendix B]{IOTW2} for discussion of nondeterminantal weights.)
}

{
When we pass from injective resolution of the target to projective resolution of the source, we go from semi-invariant conditions on the positive simplex to stability conditions on the entire plane.
We take two semi-invariants on the quiver $Q:1\ot 2\ot 3$ as example. 

\begin{eg}\label{eg: A3 injective semi-invariants}
Take a $KQ$-module $V$ considered as a representation of the quiver $Q:1\ot 2\ot 3$:
\[
	V_1\xleftarrow f V_2\xleftarrow g V_3
\] with dimension vector $\undim V=(v_1,v_2,v_3)$ where $v_i=\dim V_i$. We consider two semi-invariants on the representation space $\Hom_K(V_2,V_1)\times \Hom_K(V_3,V_2)$.
\begin{enumerate}
\item $\varphi(f,g)=\det f$ with $wt(\varphi)=(1,-1,0)$ and
\item $\psi(f,g)=\det(fg)$ with $wt(\psi)=(1,0,-1)$.
\end{enumerate}
These semi-invariants are not always defined. The set of dimension vectors of $V$ for which these semi-invariants are defined and nonzero form the integer points of a closed convex set which we call the \emph{domain} of the semi-invariants.
These domains are indicated in Figure \ref{fig: injective domains}. Conversion to ``projective'' semi-invariants goes like this. 

(1) $\varphi(V)$ is defined when $\Hom_\Lambda(V,I_1)\cong \Hom_\Lambda(V,I_2)$. Since $I_1\to I_2$ gives an injective co-resolution of $S_1$, this is equivalent to $\Hom_\Lambda(V,S_1)=0=\Ext^1_\Lambda(V,S_1)$. Convert to a projective presentation $Q\to P\to V$ of $V$. We get the condition $\Hom_\Lambda(P,S_1)=\Hom_\Lambda(Q,S_1)$. Equivalently, $g\cdot\undim S_1=0$ where $g$ is the $g$-vector of $V$ ($g(V)=\undim P/rad\,P-\undim Q/rad\,Q$). In terms of $g$-vectors, 
\[
D(S_1)=\{g\in\RR^n\,|\,g\cdot\undim S_1=0\}
\]
is the domain of the projective semi-invariant corresponding to $\varphi$.

(2) $\psi(V)$ is defined when $hom(V,I_1)=hom(V,I_3)\le hom(V,I_2)$ where $hom(X,Y)$ is the dimension of $\Hom_\Lambda(X,Y)$. By the injective copresentations $P_2\to I_1\to I_3$ and $S_1\to I_1\to I_2$, this is equivalent to $\Hom_\Lambda(V,P_2)=0=\Ext^1_\Lambda(V,P_2)$ and $\Hom_\Lambda(V,S_1)=0$. Converting to a projective presentation $Q\to P\to V$, these conditions are $hom(Q,P_2)=hom(P,P_2)$ and $hom(P,S_1)\le hom(Q,S_1)$ or, equivalently, $g\cdot\undim P_2=0$ and $g\cdot \undim S_1\le0$ where $g$ is the $g$-vector of $V$. Thus
\[
	D(P_2)=\{g\in\RR^n\,|\,g\cdot\undim P_2=0,\ g\cdot\undim S_1\le0\}
\]
is the domain of the projective semi-invariant corresponding to $\psi$.
\end{eg} 
} 

{
\begin{figure}[htbp]
\begin{center}
\begin{tikzpicture}

\coordinate (DF) at (1,0.5);
\coordinate (DC) at (-.1,1.7);

\draw[blue] (DF) node[above]{\tiny$D(\varphi)$};
\draw[red] (DC) node[right]{\tiny$D(\psi)$};

\coordinate (S1) at (-2,0);
\coordinate (S2) at (0,3);
\coordinate (S3) at (2,0);
\coordinate (P2) at (-1,1.5);
\coordinate (I2) at (1,1.5);
\coordinate (P3) at (0,1);
\draw (S1)--(I2);
\draw[very thick,blue] (S3)--(P2);
\draw (S2)--(S3)--(S1)--(S2);
\draw[very thick,red] (S2)--(P3);
\draw (S1) node[left]{\small$S_1$};
\draw (P2) node[left]{\small $P_2$};
\draw (I2) node[right]{\small $I_2$};
\draw (S2) node[right]{\small $S_2$ };
\draw (S3) node[right]{\small $S_3$ };
\draw (P3) node[below]{\small $P_3$};
	\draw[fill] (S1) circle[radius=2pt];
	\draw[fill] (S2) circle[radius=2pt];
	\draw[fill] (S3) circle[radius=2pt];
	\draw[fill] (P2) circle[radius=2pt];
	\draw[fill] (I2) circle[radius=2pt];
	\draw[fill] (P3) circle[radius=2pt];
\end{tikzpicture}
\caption{The domain {\color{blue}$D(\varphi)$ is shown in blue}. These are the points with coordinates $(x,x,y)$. {\color{red}$D(\psi)$ is in red.} These are the points $(x,y,x)$ where $x\le y$. The figure shows the normalized dimension vectors: normalized by dividing by the sum $\sum v_i$.}
\label{fig: injective domains}
\end{center}
\end{figure}
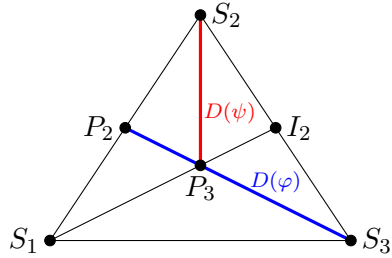
} 

\begin{figure}[htbp]
\begin{center}
\begin{tikzpicture}[scale=.6]
\coordinate (P1) at (-2,0);
\coordinate (P1p) at (-2,0.2);
\coordinate (P3) at (2,0);
\coordinate (P3p) at (2,0.2);
\coordinate (P2) at (0,3.464);
\coordinate (nP1) at (4,3.464);
\coordinate (nP3) at (-4,3.464);
\coordinate (nP2) at (0,-3.464);
\coordinate (nP1r) at (4,3.55);
\coordinate (nP3r) at (-4,3.55);
\coordinate (nP2r) at (0,-3.55);

\coordinate (S2) at (2,4);
\coordinate (S2r) at (2.2,4);
\coordinate (S3) at (1.45,-2);
\coordinate (S3r) at (1.6,-2);
\coordinate (X) at (4,0);
\coordinate (Y) at (4.5,2);

\coordinate (I2) at (3.6,1.7);
\coordinate (I2r) at (3.7,1.7);

\draw (S2r) node[above]{\small $S_2$};
\draw (S3r) node[below]{\small $S_3$};
\draw (I2) node[right]{\small $I_2$};

\begin{scope}
\clip (P2) circle[radius=4cm];
\draw[very thick,red] (-.85,2) circle[radius=3.464cm];
\draw[red] (-.2,5.9) node{$D(P_2)$};
\end{scope}

\begin{scope}
\clip (P3) circle[radius=4cm];
\draw[thick] (1.15,1.464) circle[radius=3.464cm];
\end{scope}

\draw[thick] (S3)..controls (X) and (Y)..(S2);

\foreach \x in {P1,P2,P3,nP1,nP2,nP3,S2,S3,I2}
\draw[fill] (\x) circle[radius=1mm];
\draw[very thick,blue] (P1) circle[radius=4cm];
\draw[blue] (-6,0) node[left]{$D(S_1)$};
\draw[thick] (P2) circle[radius=4cm];
\draw[thick] (P3) circle[radius=4cm];
\draw (P1p) node[right]{\small$S_1$};
\draw (P3p) node[left]{\small$P_3$};
\draw (P2) node[below]{\small$P_2$};

\draw (nP2r) node[below]{\small$P_2[1]$};
\draw (nP1r) node[right]{\small$P_1[1]$};
\draw (nP3r) node[left]{\small$P_3[1]$};

\draw (6,0) node[right]{\small$D(S_3)$};
\draw (4.2,6) node{\small$D(S_2)$};
\draw (2.5,-1.9) node[right]{\small$D(I_2)$};
\draw (2.5,0) node[right]{\tiny$D(P_3)$};

\end{tikzpicture}
\caption{This is the $g$-vector fan \cite{FZ4, PPPP} for $A_3$ in $\RR^3$ intersected with the unit sphere $S^2$ and stereographically projected to the plane. Thus, hyperplanes, such as $D(S_1)$ become circles in this figure. The half-plane $D(P_2)$ becomes a semi-circle. This figure is homeomorphic to the figure in Example \ref{eg: A3 cluster tilting objects}. We call this the \emph{semi-invariant picture} for $A_3$. This is also the representation theoretic version of the ``second order Steinberg relation'' shown in Figure \ref{Fig: A3 picture with xij}.}
\label{fig: projective si domains}
\end{center}
\end{figure}
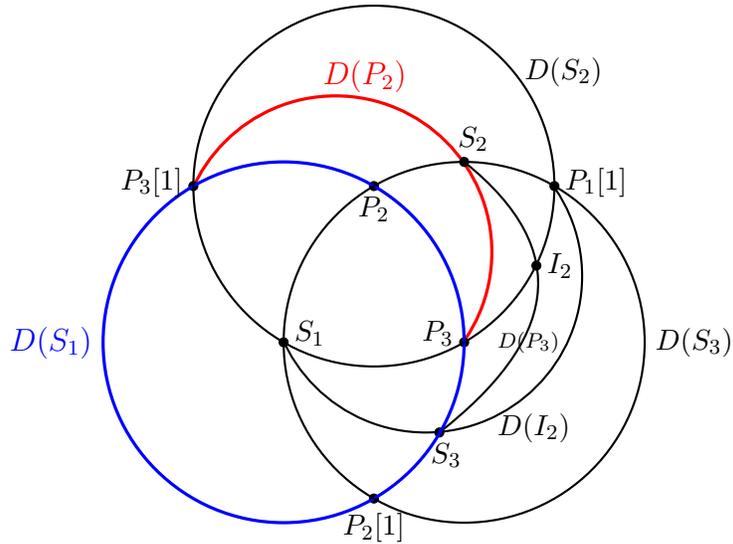

The formula for the $g$-vector \emph{semi-invariant domains} are:
\begin{equation}\label{eq: def of D(M)}
        D(M)=\{g\in\RR^n\,|\, g\cdot \undim M=0,\ g\cdot \undim M'\le0\ \forall M'\subset M\}
\end{equation}
for all exceptional modules $M$. For example, $D(P_3)$ is the set of all $g$ so that $g_1+g_2+g_3=0$, $g_1+g_2\le0$ and $g_1\le0$ since $P_2,S_1$ are the submodules of $P_3$.

{
In general (finite type) we have the following.
\begin{thm}\cite{IOTW}\label{thm: semi-invariant triangulation of sphere}
In the triangulation of $S^{n-1}$ given by Theorem \ref{thm: cluster-tilting triangulation}, the $n-2$ skeleton is the union of all semi-invariant domains $D(M)$. In particular, every $n-2$ simplex is contained in some $D(M)$ defined in \eqref{eq: def of D(M)} above.
\end{thm}
}

\section{Picture groups and picture spaces}\label{sec 4: picture group} 

{Given a picture $P$ for some group $G=\left<\cX\,|\,\cY\right>$, the \emph{picture group} of $P$ is the group $G(P)=\left<\cX_0\,|\,\cY_0\right>$ where $\cX_0\subset \cX$ is the set of edge labels which actually occur in the picture $P$ and $\cY_0\subset \cY$ is the set of relations which label the vertices of $P$.
}
{
\begin{eg}\label{eg: examples of picture groups}
\begin{enumerate}
\item The picture group for Example \ref{eg: Z3 picture} is the group $\ZZ_3=\left<x\,|\,x^3\right>$.
\item For Example \ref{eg: 3 circles}, the picture group is $\ZZ^3$ since all three generators and all three commutator relations are in the picture.
\item The picture group for the left hand picture in Figure \ref{fig: two second order relations} is 
\[
	G(P)=\left< x,y,z,w\,|\, [x,w], [y,w], [z,w], [x,y]z^{-1}\right>=\left< x,y,w\,|\, [x,w], [y,w]\right>=F_2\times \ZZ.
\]
\item The picture group for Figure \ref{fig: projective si domains} is the ``Stasheff group'' $G(A_3)$ from Definition \ref{def: Stasheff gp} below.
\item There is an exceptional case of left hand picture in Figure \ref{fig: two second order relations} when $x_{ik}^{uv}=x_{\ell m}^w$. This is the Heisenberg group shown in Figure \ref{fig: Heisenberg group}. It is one of the examples considered in \cite{IgOrr}.
\end{enumerate} 
\end{eg}
}

{
\begin{figure}[htbp]
\begin{center}
\begin{tikzpicture}
{
\begin{scope}
\clip rectangle (2,1.3) rectangle (0,-1.3);
\draw[thick] (0,0) ellipse [x radius=1.5cm,y radius=1.21cm];
\end{scope}
\begin{scope}[xshift=.87mm]
		\draw[thick] (1.3,.4) node[left]{$z$}; 
\end{scope}
\begin{scope}[xshift=-.7cm]
	\draw[thick] (0,0) circle[radius=1.4cm];
		\draw (-1,1.1) node[left]{$x$}; 
\end{scope}
\begin{scope}[xshift=.7cm]
	\draw[thick] (0,0) circle[radius=1.4cm];
		\draw (1,1.1) node[right]{$y$}; 
\end{scope}
\begin{scope}[yshift=-1.35cm]
	\draw[thick] (0,0) circle[radius=1.4cm];
	\draw (1.15,-.8) node[right]{$z$}; 
\end{scope}
\begin{scope}[xshift=12mm,yshift=-7mm]
\draw[fill,white] (0,0) circle[radius=4mm];
\draw[thick] (-.3,-.27)..controls (0,0) and (0,.1)..(-.21,.34);
\draw[thick] (.17,-.37)--(.23,.34);
\end{scope}
}
\end{tikzpicture}
\caption{The picture group of this picture is the Heisenberg group $H=\left< x,y,z\,|\, [x,z],[y,z],[x,y]z^{-1}\right>$. This is a torsion-free nilpotent group since $[x,y]=z$ is central making $H$ a central extension of $\ZZ$ by $\ZZ^2$. This is isomorphic to the unipotent upper triangular matrix group $U(A_2)$ discussed below.}
\label{fig: Heisenberg group}
\end{center}
\end{figure}
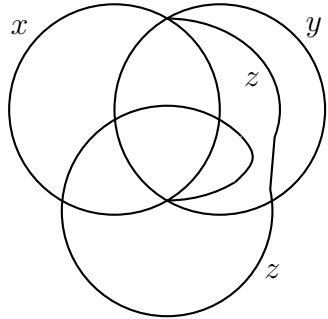
}

These are examples of picture groups of 2-dimensional pictures. For more general picture groups, we need to discuss the higher dimensional pictures given by Theorems \ref{thm: cluster-tilting triangulation} and \ref{thm: semi-invariant triangulation of sphere}. First, we describe the picture groups $G(A_n)$ of the quiver $A_n$. 

\subsection{Loday's construction}\label{ss: Loday}

For the quiver of type $A_n$ with straight orientation
\begin{equation}\label{eq: An straight}
   Q:\quad 1\ot 2\ot \cdots\ot n.
\end{equation}
the picture group was first defined by Loday who called the picture group the ``Stasheff group'' and denoted it by $Sta_{n+1}$. (Thus, $Sta_n=G(A_{n-1})$.) Loday also constructed the picture space and observed that its fundamental group is the Stasheff group.

\begin{defn}\cite{Loday}\label{def: Stasheff gp}
The \emph{Stasheff group} $Sta_n$ is the group with generators $x_{ij}$ for $1\le i<j\le n$ with relations:
\begin{enumerate}
\item $[x_{ij},x_{jk}]=x_{ik}$ for $i<j<k$ where $[a,b]:=b^{-1}aba^{-1}$ (so, if $[a,b]=c$ then $ab=bca$)
\item $[x_{ij},x_{k\ell}]=1$ if $i,j,k,\ell$ are distinct and either $(i,j)$ and $(k,\ell)$ are disjoint intervals or one is contained in the other.
\end{enumerate} 
\end{defn}

For example, $Sta_3$ has only three generators $x_{12},x_{23}$ and $x_{13}$ with only one relation $x_{13}=[x_{12},x_{23}]$. So, $Sta_3$ is the free group on two generators. Loday obtained his Stasheff group by putting labels on edges and faces of the Stasheff associahedron \cite{Stasheff} and collapsing all vertices to one vertex, identifying edges with the same labels, taking faces to be relations and identifying faces with the same relations. For example, $K_3$ is a pentagon with oriented edges labeled $12,23,13$ as shown in Figure \ref{fig: K3 with labels}.
{
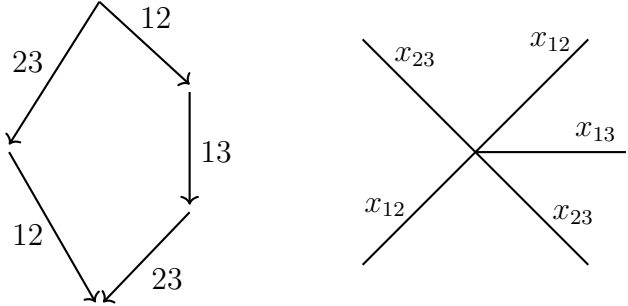
\begin{figure}[htbp]
\begin{center}
\begin{tikzpicture}
\coordinate (A) at (0,4);
\coordinate (B) at (-1.2,2);
\coordinate (AB) at (-.6,3.2);
\coordinate (BC) at (-.6,.9);
\coordinate (AD) at (.75,3.5);
\coordinate (DE) at (1.2,2);
\coordinate (EC) at (.9,.6);

\draw (AB) node[left]{23};
\draw (BC) node[left]{12};
\draw (AD) node[above]{12};
\draw (DE) node[right]{13};
\draw (EC) node[below]{23};

\coordinate (Bp) at (-1.2,2.1);
\coordinate (CL) at (-.05,0);
\coordinate (CR) at (0.05,0);
\coordinate (D) at (1.2,2.8);
\coordinate (Dp) at (1.2,2.9);
\coordinate (E) at (1.2,1.2);
\coordinate (Ep) at (1.2,1.3);
\coordinate (C) at (0,0);

\draw[thick,->] (A)--(Bp);
\draw[thick,->] (B)--(CL);
\draw[thick,->] (E)--(CR);
\draw[thick,->] (D)--(Ep);
\draw[thick,->] (A)--(Dp);

{
\begin{scope}[xshift=5cm,yshift=2cm]
  \draw[thick]  (-1.5,-1.5)--(1.5,1.5);
  \draw[thick]  (1.5,-1.5)--(-1.5,1.5);
  \draw[thick]  (0,0)--(2,0);
  \draw (1,1.2) node[above]{$x_{12}$};
  \draw (-1.2,-1) node[above]{$x_{12}$};
  \draw (1.3,-1.1) node[above]{$x_{23}$};
  \draw (-.8,1) node[above]{$x_{23}$};
  \draw (1.6,0) node[above]{$x_{13}$};
\end{scope}
}

\end{tikzpicture}
\caption{On the left, Loday composes arrows right to left to get the relation $(12)(23)=(23)(13)(12)$ which corresponds to the relation $x_{12}x_{23}=x_{23}x_{13}x_{12}$ which is an equality of row operations corresponding to the dual hand slide picture on the right.
}
\label{fig: K3 with labels}
\end{center}
\end{figure}
}

Loday takes the pentagon $K_3$ and identifies all vertices and pairs of edges with the same label giving the 1-dimensional space $X(1)=\ast\cup e^1_{12}\cup e^1_{23}\cup e^1_{13}$. The pentagon is the attaching map of a 2-cell and we get
\[
    X(1)\cup e_r^2
\]
which is our picture space $X(A_2)$. This is clearly homotopy equivalent to a wedge of two circles, so $Sta_3$ is the free group with two generators.

{
\begin{figure}[htbp]

\begin{center}
\begin{tikzpicture}
\coordinate (S1) at (-2,0);
\coordinate (S2) at (0,3);
\coordinate (S3) at (2,0);
\coordinate (P2) at (-1,1.5);
\coordinate (I2) at (1,1.5);
\coordinate (P3) at (0,1);
\coordinate (nP1) at (5,3.5);
\coordinate (nP3) at (-5,3.5);
\coordinate (nP2) at (0,-4);

\coordinate (a0) at (-4.5,-1.25);
\coordinate (a1) at (-9/4,0);
\coordinate (a2) at (0,-1);
\coordinate (a3) at (9/4,0);
\coordinate (a4) at (9/2,-5/4);
\coordinate (b1) at (-7/4,3/2);
\coordinate (b2) at (-1,1);
\coordinate (b3) at (0,1/4);
\coordinate (b4) at (1,1);
\coordinate (b5) at (7/4,3/2);
\coordinate (c1) at (-5/4,9/4);
\coordinate (c2) at (-1/2,7/4);
\coordinate (c3) at (1/2,7/4);%
\coordinate (c4) at (5/4,9/4);%
\coordinate (d1) at (0,13/4);
\coordinate (d2) at (0,5);

\draw[thick, ->] (b3)--(a2);
\draw[thick, ->]  (b5)--(a3);

\coordinate (a2b3) at (-.1,-2.8/8);
\coordinate (a3b5) at (1.9,3.2/4);

\draw (a2b3) node[right]{24};
\draw (a3b5) node[right]{24};

\coordinate (c34) at (8/8,8.3/4);
\coordinate (b45) at (12/8,5.3/4);

\draw (c34) node[below]{\tiny 14};
\draw (b45) node[below]{\tiny 14};

\draw[thick,->] (c3)--(c4);
\draw[thick,->] (b4)--(b5);

\coordinate (a12) at (-9/8,-1/2);
\coordinate (b12) at (-12/8,5/4);
\coordinate (c12) at (-8/8,2);
\coordinate (d1c4) at (6.5/8,10.5/4);
\coordinate (a34) at (27/8,-5/8);

\draw (a34) node[above]{34};
\draw (a12) node[below]{ 34};
\draw (b12) node[below]{\tiny34};
\draw (c12) node[below]{\tiny34};
\draw (d1c4) node[above]{34};

\coordinate (d12) at (0,4.1);
\coordinate (c4b5) at (5.5/4,16/8);
\coordinate (c3b4) at (4/8,10/8);
\coordinate (b23) at (-5/8,5/8);
\coordinate (ab1) at (-1.9,7/8);

\draw (d12) node[left]{23};
\draw (c4b5) node[right]{\tiny23};
\draw (c3b4) node{\tiny23};
\draw (b23) node[below]{23};
\draw (ab1) node[left]{23};

\draw[thick,->] (d1)--(d2);
\draw[thick,->]  (c4)--(b5);
\draw[thick,->]  (c3)--(b4);
\draw[thick,->]  (b2)--(b3);
\draw[thick,->]  (b1)--(a1);
\draw[thick,->] (a2)--(a1);
\draw[thick,->]  (b2)--(b1);
\draw[thick,->]  (c2)--(c1);
\draw[thick,->]  (c4)--(d1);
\draw[thick,->]  (a3)--(a4);

\coordinate (bc1) at (-11/8,16/8);
\coordinate (bc2) at (-3.5/4,10/8);
\coordinate (b34) at (.65,6/8);
\coordinate (a23) at (9/8,-1/2);
\coordinate (a01) at (-27/8,-5/8);
\coordinate (c23) at (0,7/4);
\coordinate (cd1) at (-7/8,10.5/4);

\draw[red] (c23) node[above]{13};
\draw[red] (cd1) node[above]{13};
\draw[blue] (bc1) node[left]{\tiny12};
\draw[blue] (bc2) node[right]{\tiny12};
\foreach \x in{b34,a23,a01}
\draw[blue] (\x) node[below]{12};

\draw[blue,very thick, ->] (b1)--(c1);
\draw[blue,very thick, ->]  (b2)--(c2);
\draw[blue,very thick, ->]  (b3)--(b4);
\draw[blue,very thick, ->] (a2)--(a3);
\draw[blue,very thick, ->] (a1)--(a0);
\draw[red,very thick, ->] (c2)--(c3);
\draw[red,very thick, ->] (c1)--(d1);
\draw (nP1)--(S1);
\draw[blue] (nP3)--(nP2)--(S3)--(nP3);
\draw (nP3)--(S1)--(nP2)--(S3)--(nP1)--(S2);
\draw (nP3)--(nP1)--(nP2);
\draw (S2)--(S3)--(S1)--(S2);
\draw[red] (nP3)--(S2)--(P3);
\draw (nP1) node[right]{$P_1[1]$};
\draw (nP3) node[left]{$P_3[1]$};
\draw (nP2) node[right]{$P_2[1]$};
\end{tikzpicture}
\caption{This shows the associahedron $K_4$ with 14 vertices which are in the centers of the 14 triangles in the dual picture which is the semi-invariant picture from Example \ref{eg: A3 cluster tilting objects}. Loday put labels on edges in $K_4$ to indicate which operations connect the binary trees at the two endpoints. Loday intended his associahedron picture to be dual to what he called ``Igusa's picture'' which is the ``second order Steinberg relations'' shown in Figure \ref{Fig: A3 picture with xij} with $x_{ij}$ dual to $(ij)$ in Loday's picture.
}
\label{fig: K4 with labels}
\end{center}
\end{figure}
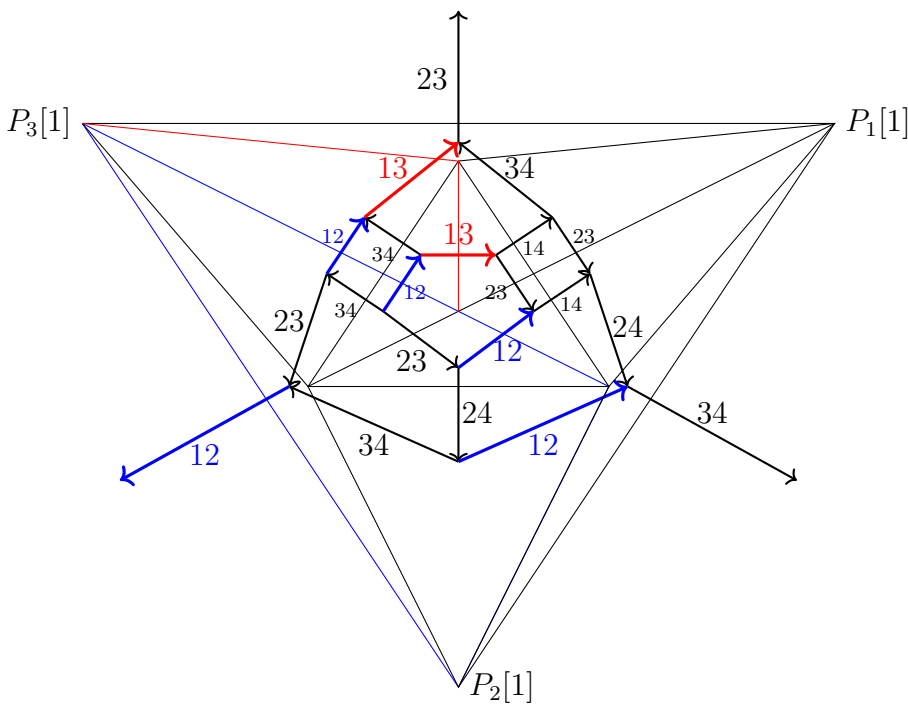
}
The next case may be more instructive. In \cite{Loday}, Loday draws the associahedron $K_4$ shown in Figure \ref{fig: K4 with labels} but without the dual picture superimposed. To each of the 14 vertices he associates a rooted binary tree with 4 internal nodes (and 5 leaves). Each edge corresponds to a mutation of a tree by switching an edge from left to right, so that the top vertex changes from vertex $i$ to vertex $j$. Then the relative heights of vertices $i,j$ switch order and the moment they are at the same height is the domain of the semi-invariant. If we turn the binary tree upside down, we see that the heights of the intermediate vertices must be larger when the heights of vertex $i$ and vertex $j$ become equal. And this is exactly the condition under which the determinant of $V_j\to V_i$ is defined. So, the duality makes sense.

Note that the faces of the associahedron are all squares and pentagons. These give the commutativity relation and the 5-term relation in Definition \ref{def: Stasheff gp}. There are 3 pairs of faces with the same relations labeling them. They are at the vertices $P_i[1]$ and $P_i$. Loday defines the space $X(2)$ to be the 2 dimensional cell complex given by these  edges and faces:
\[
    X(2)=(\ast\vee_{x\in\cX}e_x^1 )\cup_{y\in\cY}e_y^2.
\]
The semi-invariant picture \ref{fig: projective si domains} gives a surjective map
\[
    f:S^2\to X(2).
\]
The mapping cone $C(f)=X(2)\cup_f e^3$ is our picture space $X(A_3)$. Loday considered this and asked if it was a $K(\pi,1)$. His paper went on to say, if this space still has higher homotopy groups, we should add higher dimensional cells and call these ``higher homotopy syzygies''.

For larger $n$, Loday defined a space $X(n)$ to be the $n$-dimensional cell complex given by making identifications on the boundary sphere of the $n+1$ dimensional associahedron $K_{n+1}$. He noted that the associahedron itself could be attached as a final $n+1$ cell to make what we call the picture space: 
\[
X(A_n)=X(n)\cup K_{n+1}.
\]
Loday obverved that the 2-dimensional faces of any associahedron $K_n$ are squares and pentagons and these give exactly the relations of the Stasheff group.

\subsection{Picture group for Dynkin quivers}\label{ss:picture grp for Dynkin}

{ 
Loday's construction is very similar to the construction of the picture group and picture space for an arbitrary hereditary algebra $\Lambda$ of finite representation type. The only difference is that we get not just squares and pentagons, but also hexagons and octagons. These are the 2-dimensional faces of the ``generalized associahedron'' obtained as the dual of the semi-invariant triangulation of the sphere $S^{n-1}$ together with the disk $D^n$. The hexagons come from algebras whose underlying valued quivers are of type $B_n,C_n,F_4$. Octagons come from $G_2$ (and $G_2\coprod Q$ for any Dynkin quiver $Q$). The disjoint union of quivers gives the product of algebras. But the hexagon is not a regular hexagon. (See also \cite{Stella}.)

\begin{eg}\label{eg: B2 quiver to hexagon}
Figure \ref{fig: B2 example} shows the process for the algebra of type $B_2$ with valued quiver
\[
	1\xleftarrow{(1,2)}2\qquad \text{or}\qquad \RR\leftarrow \CC.
\]
There are 4 positive roots: $\beta,\alpha+\beta,2\alpha+\beta,\alpha$. We like to write the order of roots as $a\alpha+b\beta$ in the increasing order of the ration $a/b$: $0/1,1/1,2/1,1/0$. By Gabriel \cite{Gabriel}, the indecomposable modules over $\Lambda$ are $M_\gamma$ for all positive roots $\gamma$. Each semi-invariant domain $D(M_\gamma)$ is contained in the hyperplane $\gamma^\perp$. The formula for $D(M)$ puts $D(M_{\alpha+\beta}$ and $D(M_{2\alpha+\beta})$ on the negative side of $D(M_\alpha)$ since $M_\alpha\subset M_{\alpha+\beta},M_{2\alpha+\beta}$. This hexagon gives the 6 term relation (composing arrows left-to-right):
\begin{equation}\label{eq: hexagon relation}
	x_\alpha x_\beta= x_{\beta}x_{\alpha+\beta} x_{2\alpha+\beta}x_\alpha.
\end{equation}
\end{eg}
}

{
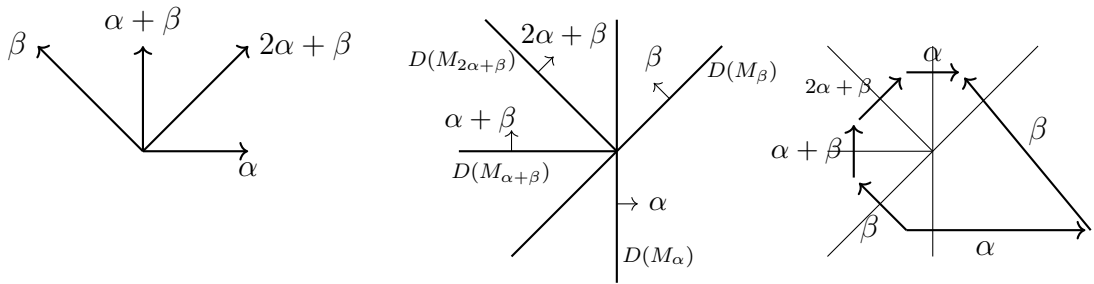
\begin{figure}[htbp]
\begin{center}
\begin{tikzpicture}[scale=.7]
\begin{scope}[xshift=-9cm]
\draw[thick,<-] (-2,2)--(0,0);
\draw[thick,<-] (2,2)--(0,0);
\draw[thick,<-] (2,0)--(0,0);
\draw[thick,<-] (0,2)--(0,0);
\draw (2,0) node[below]{$\alpha$};
\draw (2,2) node[right]{$2\alpha+\beta$};
\draw (0,2) node[above]{$\alpha+\beta$};
\draw (-2,2) node[left]{$\beta$};
\end{scope}
\begin{scope}
\draw[thick] (-2.5,2.5)--(0,0); 
	\draw[->] (-1.5,1.5)--(-1.2,1.8);
	\draw (-1,2.2) node{\small$2\alpha+\beta$};
	\draw(-1.7,1.7) node[left]{\tiny$D(M_{2\alpha+\beta})$};
\draw[thick] (2,2)--(-2,-2); 
	\draw (1.5,1.5) node[right]{\tiny$D(M_\beta)$};
	\draw[->] (1,1)--(.7,1.3);
	\draw (.7,1.3) node[above]{\small$\beta$};
\draw[thick] (0,0)--(-3,0); 
	\draw(-2.2,0)node[below]{\tiny$D(M_{\alpha+\beta})$};
	\draw[->] (-2,0)--(-2,.4);
	\draw(-1.7,.6)node[left]{\small$\alpha+\beta$};
\draw[thick] (0,2.5)--(0,-2.5);
	\draw(-.1,-2)node[right]{\tiny$D(M_\alpha)$};
	\draw[->] (0,-1)--(.4,-1);
	\draw (.4,-1)node[right]{\small$\alpha$};
\end{scope}
\begin{scope}[xshift=6cm]
\draw (-2,2)--(0,0);
\draw (2,2)--(-2,-2);
\draw (0,0)--(-2,0);
\draw (0,2)--(0,-2);
\coordinate (A) at (-.5,1.5);
\coordinate (Am) at (-.6,1.4);
\coordinate (B) at (.5,1.5);
\coordinate (Bp) at (.6,1.4);
\coordinate (F) at (-1.5,.5);
\coordinate (Fp) at (-1.4,.6);
\coordinate (E) at (-1.5,-.5);
\coordinate (Em) at (-1.4,-.6);
\coordinate (D) at (-.5,-1.5);
\coordinate (C) at (3,-1.5);
\coordinate (Cm) at (2.9,-1.5);
\draw[thick,->] (Fp)--(Am);
\draw[thick,->] (A)--(B);
\draw[thick,->] (C)--(Bp);
\draw[thick,->] (D)--(Em);
\draw[thick,->] (E)--(F);
\draw[thick,->] (D)--(Cm);
\draw (0,1.5)node[above]{$\alpha$};
\draw (1,-1.5)node[below]{$\alpha$};
\draw (2,-.1)node[above]{$\beta$};
\draw (-1.2,-1)node[below]{$\beta$};
\draw (-1.5,0)node[left]{\small$\alpha+\beta$};
\draw (-.95,1.2)node[left]{\tiny$2\alpha+\beta$};
\end{scope}
\end{tikzpicture}
\caption{For the Dynkin quiver $B_2$ the root system is shown on the left. With $\alpha$ being the short root and $M_\alpha$ being the simple projective, the semi-invariant domains is shown next with $D(M_{\alpha+\beta})$ and $D(M_{2\alpha+\beta})$ on the negative side of $D(M_\alpha)$. The dual hexagon is shown on the right.}
\label{fig: B2 example}
\end{center}
\end{figure}
}

{
Following \cite{ITW16}, we denote by $\Sigma(Q)$ the simplicial complex which is the semi-invariant triangulation of $S^{n-1}$ where $Q$ is the underlying valued quiver of $\Lambda$ \cite{IOTW2}\cite{DR}. Instead of attaching one $n$-cell at the end as Loday does, we take the cone of $\Sigma(Q)$. This has one $n$-simplex for every $n-1$ simplex in $\Sigma(Q)$. By Theorem \ref{thm: cluster-tilting triangulation}, the $n-1$ simplices of $\Sigma(Q)$ correspond to the cluster-tilting object of $\Lambda$. Thus, the cone $C(\Sigma(Q))$ will have one $n$-simplex for every cluster-tilting object of $\Lambda$. 

Following Loday, we take the dual of the semi-invariant picture $\Sigma(Q)$, call it $D\Sigma(Q)$. Then we identify faces of this dual which have the same ``label''. Loday called the result $X(n-1)$ since it is $n-1$ dimensional. Then he attached an $n$-cell along the map
\[
S^{n-1}=\Sigma(Q)\cong D\Sigma(Q)\xrightarrow \varphi D\Sigma(Q)/\sim=X(n-1).
\]
Then Loday's picture space is $X(n-1)\cup e^n$ which has only one $n$-cell attached to $X(n-1)$ by the ``picture''. We do something equivalent (homeomorphic). Our picture space is
\[
	X(Q)=C(\Sigma(Q))\cup_\varphi D\Sigma(Q)/\sim.
\]
This format has the advantage that it makes $X(Q)$ a union of cubes, one for every cluster-tilting object of $\Lambda$.
}

{
We need to describe the labeling of the dual of $\Sigma(Q)$. The vertices of the dual are the barycenters of the top dimensional simplices. These are unlabeled. We will identify them all to one vertex. (So, the picture space will have two vertices, the vertices of the dual sphere identified to one point and the cone point.) 

The edges of the dual of $\Sigma(Q)$ are the lines that go from the barycenters of two $n-1$ dimensional simplices which share an $n-2$ dimensional face. Recall from Theorem \ref{thm: semi-invariant triangulation of sphere} that the $n-2$ dimensional simplices form the domains of semi-invariants $D(M)$. The edge goes through the barycenter of that face. See Figure \ref{Fig: two triangles}. We use the notation $M_\alpha$ to denote the indecomposable module with dimension vector $\alpha=\undim M_\alpha$. It is well-known \cite{Gabriel} that these vectors form the set of positive roots of the root system with Dynkin diagram the underlying graph of the valued quiver $Q$ of $\Lambda$.
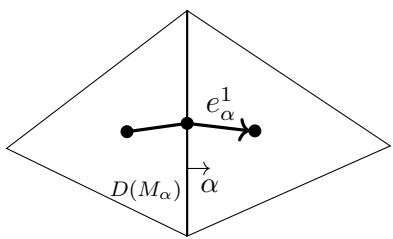
\begin{figure}[htbp]
\begin{center}
\begin{tikzpicture}[scale=1.5]
\coordinate (A) at (0,0);
\coordinate (Cp) at (0.05,.4);
\coordinate (C) at (0,1);
\coordinate (Cn) at (0,.6);
\coordinate (Cnp) at (.2,.6);
\coordinate (B1) at (-8/5,7/9);
\coordinate (B15) at (-8/15,25/27);

\coordinate (B25mm) at (1.5/5,14/15);
\coordinate (B25m) at (2.8/5,14/15);
\coordinate (B25) at (3/5,14/15);
\coordinate (B2) at (9/5,4/5);
\coordinate (Ap) at (0,2);

\draw[->] (Cn)--(Cnp);
\draw (.2,.45)node{\small$\alpha$};
\draw (Cp) node[left]{\tiny $D(M_\alpha)$};
\draw (A)--(B1)--(Ap)--(A)--(B2)--(Ap);
\draw[thick] (A)--(Ap);
\draw[very thick,->] (B15)--(C)--(B25m);
\foreach \x  in {C,B15,B25}
\draw[fill] (\x) circle[radius=1.5pt];
\draw(B25mm) node[above]{$e^1_\alpha$};
\end{tikzpicture}
\caption{The dual 1-cell $e^1_\alpha$ crosses the middle of a codimension one face contained in $D(M_\alpha)$ which is a subset of the hyperplane perpendicular to the vector $\alpha=\undim M_\alpha$. We orient $e^1_\alpha$ in the direction $\alpha$. }
\label{Fig: two triangles}
\end{center}
\end{figure}
}

{
A dual 2-cell $E(\rho)$ of $D\Sigma(Q)$ is a $k$-sided polygons centered at the barycenter of a codimension-2 simplex $\rho$ in $\Sigma(Q)$ where $k=4,5,6$, or $8$. Figure \ref{fig: dual square} below illustrates the case $k=4$ and Figure \ref{fig: B2 example} shows the case $k=6$ with $\rho$ being the center point. Looking at Figure \ref{fig: dual square} (and Figure \ref{fig: K4 with labels}) we see that the parallel sides of a dual square have the same labels. We also see that, when vertices are identified and edges with the same label are identified, any dual square will become a torus. Also, if two toruses have the same label, they are identified. For example, in Figure \ref{fig: K4 with labels} two squares are labeled with $12,34$. These are pasted together but with the opposite orientation.

}

{
\begin{figure}[htbp]
\begin{center}
\begin{tikzpicture}
\coordinate (A) at (0,0);
\coordinate (A0) at (0,-2);
\coordinate (A1) at (0,2);
\coordinate (BL) at (-1.6,-2);
\coordinate (BR) at (1.6,2);
\coordinate (CL) at (-3,0);
\coordinate (CR) at (3,0);
\coordinate (X0) at (0,0);
\coordinate (X1) at (0,0);
\coordinate (Y0) at (0,0);
\coordinate (Y1) at (0,0);

\begin{scope}
\clip (CL)--(A0)--(A1)--(CL);
\foreach \x in {0,.2,.4,.6,.8,1,1.2,1.4,1.6,1.8}
\draw[blue] (-3,\x)--(3,\x);
\foreach \x in {.2,.4,.6,.8,1,1.2,1.4,1.6,1.8}
\draw[blue] (-3,-\x)--(3,-\x);
\draw[very thick,blue] (A0)--(CL)--(A1);
\draw[fill,white] (A0)--(A1)--(BL)--(A0);
\end{scope}

\begin{scope}
\clip (BL)--(A0)--(BR)--(A1)--(BL);
\foreach \x/\y in {.2/-1.4,.4/-1.2,.6/-1,.8/-.8,1/-.6,1.2/-.4,1.4/-.2}
\draw[red] (\x,2)--(\y,-2);
\draw[very thick,red] (A0)--(BL)--(A1)--(BR)--(A0);
\draw[fill,white] (A0)--(A1)--(CR)--(A0);
\end{scope}

\begin{scope}
\clip (CR)--(A0)--(A1)--(CR);
\foreach \x in {0,.2,.4,.6,.8,1,1.2,1.4,1.6,1.8}
\draw[blue] (-3,\x)--(3,\x);
\foreach \x in {.2,.4,.6,.8,1,1.2,1.4,1.6,1.8}
\draw[blue] (-3,-\x)--(3,-\x);
\draw[very thick,blue] (A0)--(CR)--(A1);
\end{scope}

\draw[very thick] (A0)--(A1);
\draw[very thick] (A)--(1,0)--(.8,-.5)--(-.2,-.5)--(A);
\draw[thick,dashed] (A)--(-1,0)--(-1.3,-.5)--(-.2,-.5)--(A);
\draw[dashed] (-1,0)--(-.7,.5)--(1.2,.5)--(1,0)
(.22,.5)--(A);
\draw[very thick] (-.8,0)--(-1,0)--(-1.3,-.5)--(-1,-.5);

\draw (-.1,-1.7) node[right]{$\rho$};
\draw (.5,-.3) node{\tiny$E(\rho)$};

\draw[red] (-1,-2) node[below]{$D(M_\alpha)$};
\draw[red] (2.2,2) node[below]{$D(M_\alpha)$};
\draw[blue] (-3,-.2) node[below]{$D(M_\beta)$};
\draw[blue] (3,-.2) node[below]{$D(M_\beta)$};

\begin{scope}[xshift=5.5cm]
\draw[thick, red] (-.5,-1)--(.5,1);
\draw[thick,blue]  (-1,0)--(1,0);
\draw[red] (0.2,-1) node{\tiny$D(M_\alpha)$};
\draw[blue] (-1,0) node[above]{\tiny$D(M_\beta)$};
\draw (0,-2) node{$\Sigma(\cW(\rho))$};
\end{scope}

\begin{scope}[xshift=8cm]
\draw[thick,red,->] (0.2,.7)--(2,.7);
\draw[thick,red,->] (-.3,-.7)--(1.5,-.7);
\draw[thick,blue,->] (0.2,.7)--(-.3,-.65);
\draw[thick,blue,->] (2,.7)--(1.5,-.65);
\draw[red] (1,.7)node[above]{$e^1_\alpha$};
\draw[red] (.6,-.7)node[below]{$e^1_\alpha$};
\draw[blue] (0,0.1)node[left]{$e^1_\beta$};
\draw[blue] (1.75,0)node[right]{$e^1_\beta$};
\draw (0.8,0) node{\small$E(\rho)$};
\end{scope}
\end{tikzpicture}
\caption{This figure shows the case $n=4$ and $k=4$. The $n-3$ simplex $\rho$ lies in the intersection of two semi-invariant domains $\rho\subset D(M_\alpha)\cap D(M_\beta)$ of which we see the positive side. The polygon $E(\rho)$ dual to $\rho$ is the square shown on the right. This square gives the relation $x_\alpha x_\beta=x_\beta x_\alpha$ in the picture group. The picture for the wide subcategory of $\rho$ is shown in the middle.}
\label{fig: dual square}
\end{center}
\end{figure}
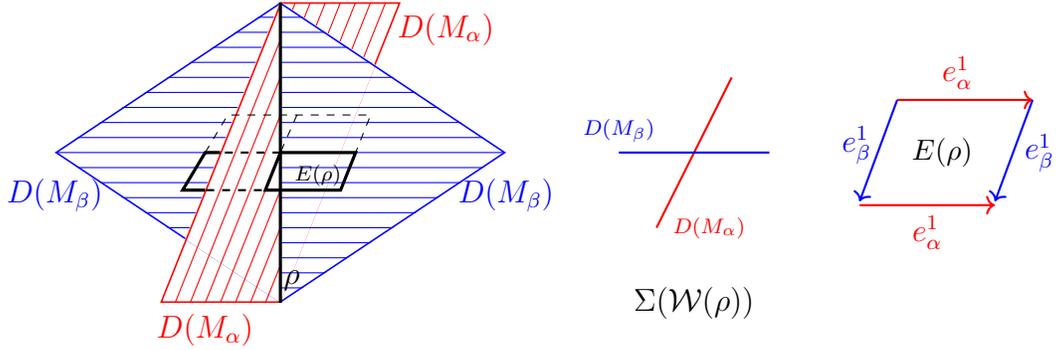
}

We need one more thing: Which pairs $\alpha,\beta$ form these dual squares? We go back to Loday's construction in order to generalize it. Let $Q$ be the quiver of type $A_n$ shown in \eqref{eq: An straight}. The path algebra $KQ$ has projective modules $P_j$ given by $K$ at each vertex $i\le j$ and 0 elsewhere. So, the support of $P_j$ is $(0,j]$. Then, all indecomposable modules are given by $M_{ij}=P_j/P_i$ where $P_0:=0$. This has support $(i,j]$. In particular, $P_j=M_{0j}$ and the simple module $S_j$ is $M_{j-1,j}$. Note that $M_{ij}$ is a submodule of $M_{k\ell}$ if and only if $i=k$ and $j\le\ell$. Similarly $M_{ij}$ is a quotient of $M_{k\ell}$ if and only if $j=\ell$ an $k\le i$. This implies that $M_{ij},M_{k\ell}$ are Hom-orthogonal, i.e.,
\[
    \Hom_\Lambda(M_{ij},M_{k\ell})=0=\Hom_\Lambda(M_{k\ell},M_{ij})
\]
if and only if the half open intervals $(i,j]$ and $(k,\ell]$ are either disjoint or $i<k<\ell<j$ or $k<i<j<\ell$. 

Even if $M_{ij},M_{k\ell}$ are Hom-orthogonal, they can still extend each other if $j=k$ or $i=\ell$. However, both cannot happen. For example, if $j=k$, we have an exact sequence $M_{ij}\to M_{i\ell}\to M_{j\ell}$ and no extension the other way. Thus, $M_{ij},M_{k\ell}$ are Hom and Ext orthogonal if and only if the closed intervals $[i,j]$ and $[k,\ell]$ are either disjoint or one is contained in the interior of the other. This is exactly condition (2) in Definition \ref{def: Stasheff gp}. In general we have the following where we use the following notation. We say that the roots $\alpha,\beta$ are \emph{hom-orthogonal} if $M_\alpha,M_\beta$ are Hom-orthogonal. Similarly $\alpha,\beta$ are \emph{ext-orthogonal} if $M_\alpha,M_\beta$ are Ext-orthogonal.

\begin{lem}\cite{ITW16}\label{lem formula for square}
    In the dual complex $D(\Sigma(Q))$, $e_\alpha^1,e_\beta^1$ form the boundary of a dual square if and only if $\alpha,\beta$ are both hom- and ext-orthogonal.
\end{lem}

For other dual polygons, we refer to the examples of the pentagons in Loday's figure \ref{fig: K4 with labels} and hexagons given by Example \ref{eg: B2 quiver to hexagon} and Figure \ref{fig: B2 example}. The octagon is similar. These give the relations in the picture group.

\begin{defn}\label{def: picture group in general}
    The \emph{picture group} of $Q$ is the group with generators $x_\alpha$ for all positive roots $\alpha$ of $Q$ with relations given by the 2-dimensional faces of the dual complex $D(\Sigma(Q))$.
\end{defn}

\begin{thm}\cite{ITW16}\label{thm: presentation of the picture group}
The \emph{picture group} of $Q$ (also called the picture group of $\Lambda$) is the group given by generators and relations as follows.
\begin{enumerate}
    \item The generators are symbols $x_\alpha$ for all positive roots $\alpha$ of $Q$. (Thus $M_\alpha$ are all the indecomposable $\Lambda$-modules.)
    \item When $\alpha,\beta$ are hom-orthogonal roots and $\Ext^1_\Lambda(M_\beta,M_\alpha)=0$ then $x_\alpha x_\beta$ is the product of all $x_\gamma$ where $\gamma$ runs over all root $a\alpha+b\beta$ which positive integer linear combinations of $\alpha$ and $\beta$ in increasing order of the ratio $a/b$. This gives
    \begin{enumerate}
        \item Square relations: $x_\alpha x_\beta=x_\beta x_\alpha$
        \item Pentagon relations: $x_\alpha x_\beta=x_\beta x_{\alpha+\beta} x_\alpha$
        \item Hexagon relations: $x_\alpha x_\beta=x_\beta x_{\alpha+\beta}x_{2\alpha+\beta} x_\alpha$
        \item Octagon relations: $x_\alpha x_\beta=x_\beta x_{\alpha+\beta}x_{3\alpha+2\beta} x_{2\alpha+\beta}x_{3\alpha+\beta}x_\alpha$
    \end{enumerate}
\end{enumerate}
\end{thm}

\begin{cor}\label{cor: Sta(n)=G(An)}
    The Stasheff group $Sta_n$ is isomorphic to the picture group of $A_{n-1}$ with linear orientation.
\end{cor}

{

\subsection{Higher dimensional cells of the picture space}\label{ss:higher dim cells of pic space}
The construction of the picture space $X(Q)$ will be by induction (recursion) on the size of $Q$. The recursive construction requires wide subcategories.
\begin{defn}\label{def: wide subcat}
    The \emph{wide subcategory} $\cW(g)$ for any point $g\in \RR^n$ is the full subcategory of $mod$-$\Lambda$ consisting of all modules $M$ so that $g\in D(M)$ where we recall that $D(M\oplus N)=D(M)\cap D(N)$. The wide subcategory for a simplex $\rho$ in $\Sigma(Q)$ is defined to be the wide subcategory for the barycenter of $\rho$. Thus $\cW(\rho)$ contains all $M$ so that $\rho\subset D(M)$.
\end{defn}

It is well-known \cite{IngTh} that, in finite type, wide subcategories $\cW$ of $mod\text-KQ$ are isomorphic to $mod\text-KQ'$ where $Q'$ is another quiver of finite type which is smaller than $Q$. Thus, by induction, we have a picture space for $\cW$ given by
\[
    X(\cW):=X(Q').
\]

\begin{eg}\label{examples of wide sub cats}
Some examples of $\cW(\rho)$:
\begin{enumerate}
    \item If $\rho$ is an $(n-2)$-simplex in $\Sigma(Q)$ then $\rho\subset D(M)$ and $\cW(\rho)=ad(M)$, the additive category generated by $M$.
    \item For $\rho$ as in Figure \ref{fig: dual square}, $\cW(\rho)$ is consists of $M_\alpha,M_\beta$ and direct sums.
    \item For $g=0$, $\cW(0)=mod\text-\Lambda$ since $0\in D(M)$ for all $M$.
\end{enumerate}
\end{eg} 
}

{
In \cite{ITW16} we proved the following where $E(\rho)$ denotes the polyhedron dual to $\rho$
\begin{thm}\label{thm: dual cell to rho}
    For every simplex $\rho$ in $\Sigma(Q)$, the dual polyhedron of $\rho$ is isomorphic to the picture space of the wide subcategory for $\rho$:
    \[
        E(\rho)\cong X(\cW(\rho)).
    \]
\end{thm}

The picture space $X(Q)$ is given by identifying all dual cells $E(\rho)$ with the same wide subcategory $\cW(\rho)$. Thus the picture space $X(Q)$ is a cell complex with one cell for every wide subcategory of $mod\text-\Lambda$ including the zero subcategory. 

In \cite{ITW16} we also gave a convoluted argument that the picture space was a $K(\pi,1)$ and therefore could be used to compute the homology of the picture group. The referee did not appreciate this indirect argument and we hope that the explanation we give below is more enlightening.
}

{
Our motivation for computing the cohomology of $G(A_n)$, the Stasheff group was very similar to Loday's. Loday was inspired by the work of Kapranov and Saito \cite{KS} who showed that, for the integer upper triangular matrix group $U_{n+1}(\ZZ)$, which they called $T_{n+1}(\ZZ)$, the cell structure of the classifying space $BU_{n+1}(\ZZ)$ for this group had cells which looked like Stasheff polytopes and products of them. By Volodin \cite{Volodin}, isomorphic copies of these spaces could be assembled to give a space whose homotopy groups were the algebraic K-theory of $\ZZ$. See \cite{KS},\cite{Book}. Thus, loosely speaking, Kapranov and Saito were speculating that Volodin's space for the K-theory of $\ZZ$ is assembled from pieces isomorphic to Stasheff polytopes.

The advantage of the picture space for $A_n$ is that it is very small and it is not too difficult to compute its homology. There is an obvious homomorphism from the Stasheff group $Sta_n=G(A_{n-1})$ to the unipotent group $U_n(\ZZ)$ given by sending generators $x_{ij}$ of $Sta_n$ to the elementary matrix $E_{ij}$ given by adding the $i$th column of the identity matrix $I_{n+1}$ to its $j$th column. In \cite{ITW16}, we were motivated by the hope that the entire integral homology of the group $U_n(\ZZ)$ could be captured by copies of this construction. 
}

{
The rational cohomology of the group $U_n(\ZZ)$ is known classically. For every Dynkin quiver $\Delta$ there is a unipotent group $U(\Delta)$ whose rational cohomology was computed long ago by Bott using geometry\cite{Bott} and later Kostant gave an algebraic proof\cite{Kostant}.

\begin{thm}\label{thm:Bott-Kostant}
$H^q(U(\Delta);\QQ)$ has basis elements $x_w$ for $w\in W(\Delta)$ of length $q$ where $W(\Delta)$ is the Weyl group of $\Delta$. 
\end{thm}

}

The integer homology of $U(\Delta)$ is not completely known. However, the integer cohomology of the picture group $G(A_n)$ is easy to compute assuming the picture space $X(A_n)$ is a $K(\pi,1)$ for $G(A_n)$. 

\begin{thm}[Gordana Todorov, Jerzy Weyman,I]\label{thm:H(G(An)}
    The integer homology of the picture group $G(A_n)$ for a quiver of type $A_n$ with any orientation of arrows is a free abelian group with rank given by ``ballot numbers''\cite{Carlitz}. So, the integer cohomology is also freely generated as an abelian group.
\end{thm}

The rank of $H^k(X(Q))$ for $Q$ a quiver of type $A_n$ is given as follows for $n\le 8$.
\[
\begin{array}{ccccccc}
n  & rk\,H^0& rk\,H^1& rk\,H^2& rk\,H^3& rk\,H^4\\
0  & 1 & 0\\
1  & 1 & 1\\
2  & 1 & 2 \\
3  & 1 & 3 & 2\\
4  & 1 & 4 & 5\\
5  & 1 & 5 & 9 & 5\\
6  & 1 & 6 & 14 & 14\\
7  & 1 & 7 & 20 & 28 & 14\\
8  & 1 & 8 & 27 & 48 & 42\\
\end{array}
\]

The picture groups themselves depend on the orientation. Computer calculations show that the picture groups for $A_3$ with orientation $1\to 2\leftarrow 3$ and $1\to 2\to3$ are not isomorphic. However, the integer cohomology rings of these picture groups are isomorphic for any two orientations of $A_n$. We suspect that similar statements hold for all Dynkin quivers.

{
\section{Simplices to cubes}\label{sec 5: simplices to cubes}

We summarize the results of the previous sections.

Given an hereditary algebra $\Lambda$ of finite representation type, its picture space is determined by the valued quiver $Q$ of $\Lambda$. The indecomposable $\Lambda$-modules are $M_\alpha$ where $\alpha\in \Phi^+(Q)$ runs over all positive roots of the root system of the Lie algebra associated to $Q$. For each $M_\alpha$ we have the semi-invariant domain $D(M_\alpha)\subset \RR^n$. We take the intersection with the unit sphere $S^{n-1}$ and we get the semi-invariant picture
\[
    \Sigma(Q)=\cup_\alpha D(M_\alpha)\cap S^{n-1}.
\]
This set gives the $(n-2)$-skeleton of a triangulation of $S^{n-1}$. The $(n-1)$-simplices of this triangulation are labeled with the cluster-tilting objects in the cluster category of $\Lambda$. These labels are distinct. For example, there are 14 triangles in $\Sigma(A_3)\subset S^2$. Let $C\Sigma (Q)$ be the convex hull of $\Sigma(Q)$ in $\RR^n$. This is the cone on $\Sigma(Q)$ with cone point the origin in $\RR^n$. Each $(n-1)$-simplex in $\Sigma(Q)$ gives an $n$-simplex in $C\Sigma(Q)$. For $A_3$ this space $C\Sigma(A_3)$ is a union of 14 tetrahedra.

Then, we take the dual $D\Sigma(Q)$ of the triangulation of the sphere $S^{n-1}$.
\begin{enumerate}
    \item $D\Sigma(Q)$ has one vertex for each $n-1$ simplex in $\Sigma(Q)$. These points are unlabeled.
    \item $D\Sigma(Q)$ has one oriented edge dual to each $n-2$ simplex $\lambda$ which is contained in the normally oriented $D(M_\alpha)$ for a unique $\alpha$. [The normal orientation of $D(M_\alpha)$ is given by the vector $\alpha$ which is perpendicular to $D(M_\alpha)$.] This edge is labeled $e_\alpha$.
    \item For each $k$-simplex $\rho$ in $\Sigma(Q)$, we have the dual cell $E(\rho)$ which is an $n-k-1$ cell in $D\Sigma(Q)$. This dual cell $E(\rho)$ is labeled with the wide subcategory $\cW(\rho)$ which is the class of all indecomposable $\Lambda$-modules $M$ so that $\rho\subset D(M)$.
    \item $E(\rho)$ is isomorphic to $X(\cW(\rho))$, the picture space of this wide subcategory.
\end{enumerate}

The picture space is defined to be $C\Sigma(Q)$ with all dual cells on its surface with the same labels identified. Since the identifications are all on the boundary surface $S^{n-1}$ of the $n$-disk $D^n$, the interior of the $n$-disk remains a single $n$-cell. We call this the \emph{minimal cell structure} of $X(Q)$. Sometimes it is convenient to view the $n$-disk as a union of $n$-simplices, one for every cluster tilting object of $\Lambda$. We call this the \emph{cubical cell structure} of $X(Q)$.

The minimal cell structure of $X(Q)$ has one cell for every wide subcategory $\cW$ of $mod\text-\Lambda$. For example $X(A_3)$ has 14 wide subcategories giving 14 cells: 
\begin{enumerate}
    \item The single vertex given by identifying all barycenters of all triangles in $\Sigma(A_3)$. This is for the zero wide subcategory.
    \item 6 edges $e_\alpha$ for the 6 positive roots of $A_3$. The wide subcategory is $ad(M_\alpha)$ which we denote simply as $\cW=\{M_\alpha\}$.
    \item 6 wide subcategories of rank $2$ which are $M_\alpha^\perp$, explained below and 
    \item one 3-cell by definition of the minimal cell structure. This is $\cW=mod\text-\Lambda$
\end{enumerate} 
The cubical cell structure views $X(A_3)$ as a union of 14 cubes. In this section we will explain how dualizing one face of a tetrahedron makes it a cube. Figure \ref{Fig: Simplex to cube} shows a triangle $\Delta^2$ being dualized with one vertex at each of the 7 barycenters of the nonempty faces of $\Delta^2$. In the figure we label each barycenter by the vector obtained by setting each nonzero barycentric coordinate to 1. Thus, the barycenter of the face with vertices $v_i$ has a 1 in the $i$th entry and other entries equal to zero. We then add the cone point, which we think of as the barycenter of the empty face, and label it $000$.

The result is a polyhedron with vertices having the same labels as the vertices of the cube $I^3=[0,1]^3$. Also, the polyhedron has the same edges as the cube since the dual decomposition of the simplex puts as edges the straight lines joining barycenters of faces $\rho\subset \tau$ with dimensions $k$ and $k+1$. These edges are oriented and form morphisms in a category, namely the category of faces of $\Delta^{n-1}$, including the empty face, with morphisms begin inclusion maps. I.e., this is a poset category. This is easily seen to be isomorphic to the product poset category $\cI^n=\{0<1\}^n$ which we call the \emph{cube category}. 

The cube $I^n$ is the geometric realization of the cube category $\cI^n$ and the picture space $X(Q)$ with its cubical cell structure is the union of several copies of the geometric realizations of these cube category with identifications among faces of cubes. But, we can reverse the construction. The union of geometric realizations is in fact the geometric realization of the union of cubical categories. The result is a bigger cubical category called the ``cluster morphism category''.

}

{ 
\begin{figure}[htbp]
\begin{center}
\begin{tikzpicture}[xscale=.9]
\begin{scope}[xshift=-3cm] 
\coordinate (A) at (0,0);
\coordinate (B) at (2,3);
\coordinate (C) at (4,0);
\coordinate (D) at (2,-.6);
\draw[thick] (A)--(B)--(C)--(A);
\draw (D) node{$\Delta^2$};
\draw (A) node[below]{$v_0$};
\draw (B) node[right]{$v_2$};
\draw (C) node[below]{$v_1$};
\end{scope}
\begin{scope}[xshift=3cm] 
\coordinate (A) at (0,0);
\coordinate (AB) at (1,1.5);
\coordinate (B) at (2,3);
\coordinate (B0) at (2,0);
\coordinate (BB) at (2,1);
\coordinate (C) at (4,0);
\coordinate (BC) at (3,1.5);
\coordinate (D) at (2,-.65);
\coordinate (E) at (4.8,3.1);
\draw (A)--(B)--(C)--(A);
\draw[dashed](A)--(E);
\draw[thick](E)--(C) (AB)--(BB)--(B0);
\draw[very thick] (BB)--(BC)--(B)--(E);
\draw (D) node{$C\Delta^2$};
\foreach \x in {A,B,C,E,AB,BB,B0,BC}\draw (\x) node{$\bullet$};
\draw (E) node[right]{cone point};
\draw (4.8,2.8) node[right]{\tiny$000$};
\draw [very thick,->] (E)--(3.4,3.05);
\draw [very thick,->] (B)--(2.5,2.25);
\draw [very thick,->] (BC)--(2.5,1.25);
\draw (A)node[below]{\tiny$100$};
\draw (AB)node[left]{\tiny$101$};
\draw (B)node[left]{\tiny$001$};
\draw (BC)node[right]{\tiny$011$};
\draw (2,.9)node[right]{\tiny$111$};
\draw (C)node[below]{\tiny$010$};
\draw (B0)node[below]{\tiny$110$};
\end{scope}
\end{tikzpicture}
\caption{When a $k$-simplex $\Delta^k$ is dualized, its cone $C\Delta^k$ becomes a $k+1$ cube $I^{k+1}$. The vertices of the cube are the barycenters of the faces of $\Delta^k$ including the empty face. The ``barycenter'' of the empty face is the cone point of $C\Delta^k$. In the figure above, the barycenters are labeled with cubical coordinates. E.g., $101$ are the cubical coordinates assigned to the barycenter of the face spanned by the first and third vertices of the simplex.
}
\label{Fig: Simplex to cube}
\end{center}
\end{figure}
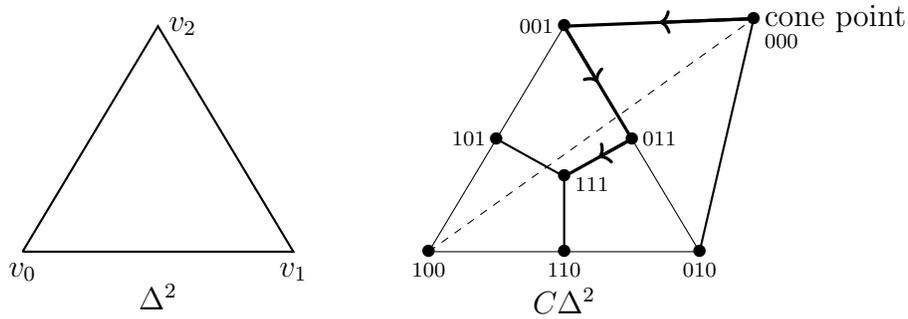
}

{\section{Cluster morphism category}\label{sec 6: signed exceptional sequences in general}

{The cube category $\cI^n$ is a graded category. Every object and morphism has a rank. Objects are sequences of $n$ zeros and ones, say $a=(a_1,a_2,\cdots,a_n)\in\{0,1\}^n$ with rank equal to the sum: $rk\, a=\sum a_i$. The rank of a morphism $a\to b$ is $rk\, b-rk\, a\ge0$ and only identity morphisms have rank 0. We are, in general, interested in factoring morphisms as compositions of morphisms of rank 1.

For the cube category $\cI^n$, there are $n!$ ways to factor the longest morphism $(0,0,\cdots,0)\to (1,1,\cdots,1)$ as a composition of rank 1 morphisms since the entries must be increased one by one in any order.
}

{Recall that we have a triangulation $\Sigma(Q)$ of the sphere $S^{n-1}$. In the cubical cell structure of the picture space $X(Q)$, the rank $k$ vertices of each cube are barycenters of $k$-simplices in $\Sigma(Q)$. Each $k$-simplex $\rho$ is labeled with the wide subcategory $\cW(\rho)$ which is given by the modules $M_\alpha$ so that $\rho\subset D(M_\alpha)$. As $k$ increases, the wide subcategories decrease since $\rho\subset \tau$ implies $\cW(\tau)\subset \cW(\rho)$. 

The cluster morphism category $\cC_\Lambda$ of $\Lambda$ is the category whose objects are finitely generated wide subcategories $\cW$ with morphisms $\cW\to\cW'$ given by choosing simplices $\rho\subset \tau$ in $\Sigma(Q)$ so that $\cW(\rho)=\cW$ and $\cW(\tau)=\cW'$. Choose an $n-1$ simplex $\sigma$ which contains $\tau$ and $\rho$. In the picture space, we have an $n$-cube corresponding to $\sigma$ and an embedding of categories
\[
	\varphi_\sigma:\cI^n\to \cC_\Lambda
\]
which maps those vertices of the cube corresponding to barycenters of $\rho,\tau$ to the objects $\cW(\rho)$ and $\cW(\tau)$ of $\cC_\Lambda$. There is only one morphism $\rho\to\tau$: the inclusion map. Thus, the morphism $\cW\to\cW'$ depends only on the choice of $\rho,\tau$. 
}

{ 
We give the original intrinsic definition of cluster morphism. We need the following well-known formula relating vertices and walls in the wall-and-chamber structure for a finite dimensional algebra. Recall that $T$ is \emph{$\tau$-rigid} if $\Hom_\Lambda(T,\tau T)=0$. In the hereditary case this is equivalent to $T$ being rigid.

\begin{lem}\label{lem: W(T) = J(T)}
(a) The $g$-vector of a $\tau$-rigid module $T$ lies in the domain $D(M)$ if and only if $\Hom_\Lambda(T,M)=0=\Hom_\Lambda(M,\tau T)$.

(b) The $g$-vector of $P[1]$ lies in $D(M)$ if and only if $\Hom_\Lambda(P,M)=0$.
\end{lem}

\begin{proof} (b) is easy. So, we assume $T$ is a $\tau$-rigid module and prove (a). Let $P\xrightarrow pQ\to T\to 0$ be a minimal projective presentation of $T$. Then, for any module $M$, we get an exact sequence:
\begin{equation}\label{eq: 4 term sequence}
0\to \Hom_\Lambda(T,M)\to \Hom_\Lambda(Q,M)\xrightarrow{p^\ast} \Hom_\Lambda(P,M)\to D\Hom_\Lambda(M,\tau T)\to 0.
\end{equation}
If $\Hom_\Lambda(T,M)=0=\Hom_\Lambda(M,\tau T)$ then $p^\ast$ is an isomorphism and
\[
	g(T)\cdot \undim M=\dim \Hom_\Lambda(Q,M)-\dim \Hom_\Lambda(P,M)=0.
\]
For $M'\subset M$, we have $\Hom_\Lambda(T,M')=0$. So, $g(M')\cdot \undim M'\le0$. So, $g(T)\in D(M)$. 

Conversely, suppose that $g(T)\in D(M)$. Then $\dim\Hom_\Lambda(T,M)=\dim\Hom_\Lambda(M,\tau T)$. So, it suffices to show $\Hom_\Lambda(T,M)=0$. Suppose not and take a nonzero $f:T\to M$. Let $M'$ be the image. By definition of $D(M)$ and \eqref{eq: 4 term sequence} for $M'$, we have $\dim \Hom_\Lambda(T,M')\le \dim \Hom_\Lambda(M',\tau T)$. So, we have a nonzero map $M'\to \tau T$. Composing with $T\onto M'$ gives a nonzero map $T\to \tau T$ contradicting that $T$ is $\tau$-rigid.
\end{proof}

Recall that, in the hereditary case, $\Hom_\Lambda(M,\tau T)=0$ is equivalent to $\Ext^1_\Lambda(T,M)=0$ by Auslander-Reiten duality. In the general case, this lemma says that $M$ is in the Jasso category $J(T\oplus P[1])=T^\perp\cap \,^\perp \tau T\cap P^\perp$ \cite{Jasso}.

Recall that the vertices of $\Sigma(Q)$ are $g$-vectors of objects $T_i$ in the cluster category and these span a simplex if and only if they are ext-orthogonal, i.e., they form a partial cluster in the cluster category of $\Lambda$.

\begin{thm}\label{thm: W(r) is right perp of T}
Suppose that $T=T_0\oplus T_1\oplus\cdots\oplus T_k$ is the partial cluster labeling the vertices of the $k$-simplex $\rho$. Then,
the wide subcategory $\cW(\rho)$ is the full subcategory of $mod\text-\Lambda$ of all modules $M$ so that $\Hom_\Lambda(T,M)=0=\Ext^1_\Lambda(T,M)$. In other words,
\[
    \cW(\rho)=T^{\perp_{01}}.
\]
\end{thm}

\begin{proof}
This follows from Lemma \ref{lem: W(T) = J(T)} and the definition: $\cW(\rho)$ is the set of all $M$ so that $\rho\subset D(M)$.
\end{proof}
}

{
In other words, $M$ is in the Hom-Ext-right perpendicular category of $T$, denoted $T^{\perp_{01}}$. This is well-known to be isomorphic to the module category of another hereditary algebra \cite{Schofield}. So, $\cW(\rho)=T^{\perp_{01}}$ has a cluster category.

\begin{defn}\cite{IT13}\label{def: cluster morphism}
If $\cW'\subset \cW$ are finitely generated wide subcategories of $mod\text-\Lambda$, a \emph{cluster morphism} $[T]:\cW\to \cW'$ is defined to be an isomorphism class of partial clusters $T$ in $\cW$ so that $\cW'=T^{\perp_{01}}\cap \cW$.
\end{defn}

\begin{eg}\label{eg: mod L to W(r)}
The vertices of a $k$-simplex $\rho$ in $\Sigma(Q)$ are the $g$-vectors of a partial cluster $T=T_0\oplus\cdots\oplus T_k$ for $\Lambda$ and this gives a cluster morphism $[T]:mod\text-\Lambda\to \cW(\rho)$. We will call this the \emph{cluster morphism determined by $\rho$}.
\end{eg}

Composition of cluster morphisms is very tricky. But, assuming that composition of cluster morphisms is well-defined, we can give the following easy description of how it works.}

{
Take cluster morphisms $[S]:\cW\to\cW'$ and $[R]:\cW'\to \cW''$. First, choose a simplex $\rho$ in $\Sigma(Q)$ so that $\cW=\cW(\rho)$. To find one such $\rho$, take the left Hom-Ext perpendicular category $\cW_0=\,^{\perp_{01}}\cW$. Choose a tilting module $T$ in $\cW_0$, i.e., a cluster-tilting object with no shifted projectives. The $g$-vectors of the components of $T$ span a simplex $\rho$ so that $\cW(\rho)=T^{\perp_{01}}=\cW$. Now compose the cluster morphisms
\[
    mod\text-\Lambda\xrightarrow{[T]}\cW=\cW(\rho)\xrightarrow{[S]}\cW'.
\]
This gives a cluster morphism $[T']:mod\text-\Lambda\to \cW'$ where $T'=T\oplus S'$. The components of $T'$ span another simplex $\tau$ containing $\rho$ and $[T']$ is the cluster morphism determined by $\tau$. This gives the following commuting diagram:
\[
\xymatrixrowsep{30pt}\xymatrixcolsep{30pt}
\xymatrix{
mod\text-\Lambda\ar[dr]^{[T']}\ar[d]_{[T]}&& \\
\cW= \cW(\rho)\ar[r]_{[S]}&\cW'=\cW(\tau)
	} 
\]
\begin{rem}\label{rem: key construction}
    This show that, given $\rho$ so that $\cW=\cW(\rho)$ and a cluster morphism $[S]:\cW\to \cW'$, there is a unique simplex $\tau$ containing $\rho$ so that the above diagram commutes where $[T]$ and $[T']$ are the cluster morphisms determined by $\rho$ and $\tau$. 

    Furthermore, we showed in \cite{IT13} that the components of $S$ and $S'$ uniquely determine each other, for example the component $S_i$ of $S$ corresponding to $S_i'$ in $S'$ is the unique exceptional module in $\cW$ so that $\undim S_i-\undim S_i$ is a linear combination of the dimension vectors of $T_j$, the components of $T$. Therefore, $[S]:\cW(\rho)\to \cW(\tau)$ is uniquely determined by commutativity of the above diagram. Since $[T]$ and $[T']$ are determined by $\rho$, $\tau$, we say that $[S]:\cW(\rho)\to \cW(\tau)$ is the \emph{morphism induced by inclusion} $\rho\subset\tau$.
\end{rem}
Repeating this process using $[R]: \cW'\to \cW''$, we get a unique simplex $\mu$ containing $\tau$ so that the following diagram commutes.
\[
\xymatrixrowsep{30pt}\xymatrixcolsep{30pt}
\xymatrix{
mod\text-\Lambda\ar[drr]^{[T'']}\ar[dr]_{[T']}\ar[d]_{[T]}&& \\
\cW= \cW(\rho)\ar[r]_{[S]}&\cW'=\cW(\tau)\ar[r]_{[R]} & \cW''=\cW(\mu)
	} 
\]
Here $\mu\supset \tau$ is uniquely determined by $[R]$ and the composition $[R]\circ[S]:\cW(\rho)\to \cW(\mu)$ is the morphism induced by the inclusion $\rho\subset\mu$.
}

{
We go one more step: choose a cluster-tilting object $M$ for $\cW''=\cW(\mu)$. Since $M$ is maximal, it gives a cluster morphism $[M]:\cW''\to 0$. As before, there is a unique maximal simplex $\sigma$ containing $\rho,\tau,\mu$ and the cluster morphisms $[T],[S],[R],[M]$ give a factorization of the cluster morphism $[T''']:mod\text-\Lambda\to \cW(\sigma)=0$ determined by $\sigma$:
\[
	[T''']=[M]\circ [R]\circ[S]\circ[T]:mod\text-\Lambda\to\cW(\rho)\to \cW(\tau)\to\cW(\mu)\to \cW(\sigma)=0
\]
Since these morphisms are all induced by inclusion of faces of $\sigma$, they are in the image of the embedding of the cube category corresponding to $\sigma$ into the cluster category
\[
	\varphi_\sigma:\cI^n\to \cC_\Lambda
\]
as discussed at the beginning of this section. This is the outline of the proof of the theorem:
\begin{thm}\cite{IT13,IT18}\label{main lemma}
The cluster morphism category is the union of images of the embeddings $\varphi_\sigma:\cI^n\to \cC_\Lambda$ and every morphism is in the image of one of these embeddings. But the cluster morphism $mod\text-\Lambda\to \cW(\sigma)=0$ is contained only in the image of $\varphi_\sigma$.
\end{thm}
}
}

\begin{cor}\cite{IT18}\label{cor: CAT0}
    The cluster morphism category is a cubical category having the property that its geometric realization is $CAT(0)$ and therefore a $K(\pi,1)$ if $\Lambda$ has finite type or is tame without tubes of rank $\ge3$.
\end{cor}

\section{Signed exceptional sequences}\label{sec 7: signed exceptional sequences}

{
As a corollary of Theorem \ref{main lemma}, we see that every cluster morphism $mod\text-\Lambda\to 0$ has exactly $n!$ factorizations into rank 1 morphisms. In fact, we showed:

\begin{lem}\cite{IT13}\label{lem: factoring cluster morphisms}
Every cluster morphism $[A]$ of rank $k$ has exactly $k!$ factorizations into a composition of $k$ morphisms of rank $1$, rank 1 morphisms being those given by a single object. These factorizations are in bijection with total orderings of the components of $A$.
\end{lem}

\begin{proof}
    The cluster morphism $[A]$ determines an embedding $\cI^k\to \cC_\Lambda$ and the $k!$ factorizations of the unique rank $k$ morphism in $\cI^k$ into rank 1 morphisms map to the factorizations of $[A]$ into rank 1 morphisms.
\end{proof}

\begin{defn}\label{def: signed exc seq}
    A \emph{signed exceptional sequence} in $\cW$ is a factorization of any cluster morphism
\[
	[X]:\cW\to \cW'
\]
into rank 1 cluster morphisms. The signed exceptional sequence is called \emph{complete} if $\cW=mod\text-\Lambda$ and $\cW'=0$.
\end{defn}

\begin{rem}\label{rem: rank one morphism}
A rank 1 morphism $[X]:\cW\to \cW'$ consists of one object $X$ in the cluster category of $\cW$. This object must be either an exceptional object (module) in $\cW$ or a shifted projective $P[1]$ where $P$ is an indecomposable projective object of $\cW$. We call this a \emph{relatively projective} object. When we factor a rank $k$ cluster morphism $\cW_0\to \cW_k$ into rank 1 morphisms we get:
\[
	\cW_0\xrightarrow{[X_1]}\cW_1\xrightarrow{[X_2]} \cW_2\cdots \xrightarrow{[X_k]}\cW_k
\]
Since $\cW_i=(X_1\oplus\cdots\oplus X_i)^{\perp_{01}}\cap \cW_0$, the sequence
\[
	X_k,X_{k-1},\cdots,X_1
\]
forms an exceptional sequence ($\Hom_\Lambda(X_i,X_j)=0=\Ext^1_\Lambda(X_i,X_j)$ for $i<j$) with the embellishment that some terms could be shifted if they are relatively projective. Thus, a signed exceptional sequence is the same as an exceptional sequence in which the relatively projective terms are allowed to be negative (shifted).
\end{rem}

\begin{thm}\cite{IT13}\label{thm: bijection ordered clusters}
There is a bijection between (complete) signed exceptional sequences and ordered clusters.
\end{thm}

We give some examples to illustrate this. We use semi-invariant pictures with chambers labeled with torsion classes. By a \emph{chamber} we mean a path component of the complement in $\RR^n$ of the union of the set of walls. In this paper, there are only finitely many walls. So, the chambers are open sets and our definition agrees with standard definition \cite{BST}. Recall that a torsion class $\cT$ in $mod\text-\Lambda$ is a class of modules closed under quotients and extensions. Torsion classes form a poset under inclusion. This poset is a lattice in the finite case. The standard construction associates a torsion class $\cT(\cU)$ to each chamber $\cU$. We give an alternate formula.

\begin{prop}\cite{GhostsI}\label{def of torsion class}
    To find the torsion class associated to a chamber, choose any $g$ in the chamber. Then $\cT$ is the collection of all modules $M$ so that $g\cdot\undim M'>0$ for all nonzero quotient modules $M'$ of $M$, including $M$.
\end{prop}

A word about terminology: In \cite{GhostsI} the condition is that $g\cdot M'>0$ for all ``weakly admissible'' quotients $M'$ of $M$. In our case, we have an abelian category ($mod\text-\Lambda$) so, all morphisms are admissible (have kernel, image and cokernel in the category). 

\begin{proof} The given formula defines a torsion class $\cT$ for any point $g\in\RR^n$. 
\begin{enumerate}
    \item Let $0\to A\to B\to C\to 0$ be a short exact sequence with $A,C\in\cT$. Then $g\cdot \undim B=g\cdot \undim A+g\cdot \undim C>0$. Also, any submodule $B'$ of $B$ is an extension of a submodule $A'$ of $A$ by a submodule $C'$ of $C$. Indeed, $A'=A\cap B'$ and $C'$ is the image of $B'$ in $C$. By the Snake Lemma, we have an exact sequence $0\to A/A'\to B/B'\to C/C'\to 0$. So, 
    \[
    g\cdot\undim (B/B')=g\cdot \undim (A/A')+g\cdot \undim (C/C')>0.
    \]
    So $B/B'\in \cT$ and $\cT$ is closed under extensions.
    \item $\cT$ is closed under taking quotients by definition since any quotient of a quotient of $X\in\cT$ is a quotient of $X$.
\end{enumerate}

Let $\cS(\cU)$ be the torsion class associated to the chamber $\cU$ according to the formula above. For $\Lambda$ of finite representation type, this is independent of the choice of $g\in\cU$ \cite[Thm 1.8]{GhostsI}. Also, $\cS(\cU)$ is different for different $\cU$ by \cite[Thm 2.6]{GhostsI}.

It is known that, when a semi-invariant wall $D(M)$ separates two chambers $\cU$, $\cU'$ with $\cU'$ on the positive side then the torsion class $\cT(\cU')$ contains the torsion class $\cT(\cU)$ and is the minimal such torsion class which also contains $M$ \cite{BCZ}. The class $\cS(\cU)$ has these properties. Indeed, it is shown in \cite[Lemma 1.9]{GhostsI} that $M$ lies in $\cS(\cU')$ but not in $\cS(\cU)$.  In \cite[Lemma 1.11]{GhostsI} it is shown that for any $X\in \cS(\cU')$ not in $\cS(\cU)$, there is an epimorphism $X\onto M$ whose kernel is shown to lie in $\cS(\cU)$ in the proof of \cite[Thm 2.1]{GhostsI}. Therefore, $\cS(\cU')$ is the smallest torsion class containing $\cS(\cU)$ which also contains $M$. Since $\cT(\cU)$ and $\cS(\cU)$ have the same properties, they are equal.
\end{proof}
}

{
Figure \ref{Fig: Torsion classes in A2} shows the semi-invariant picture and the Hasse diagram for the lattice of torsion classes for the quiver $Q: 1\to 2$. Following the notation in \cite{DIRRT} with circles and squares reversed, we circle the torsion classes and put ``brick labels'' on the edges of the Hasse diagram in boxes. They are bricks since we only need $D(M)$ when $M$ is a brick.

{
\begin{figure}[htbp]
\begin{center}
\begin{tikzpicture}
\coordinate (A) at (0,0);
\coordinate (B) at (-2,3);
\coordinate (C) at (0,6);
\coordinate (D) at (2.8,1.8);
\coordinate (E) at (2.8,4.2);
\coordinate (AB) at (-1,1.5);
\coordinate (BC) at (-1,4.5);
\coordinate (AD) at (1.4,.9);
\coordinate (DE) at (2.8,3);
\coordinate (EC) at (1.4,5.1);
\draw[thick] (A)--(B)--(C)--(E)--(D)--(A);

\foreach \y in {A,B,C,D,E}
\draw[fill,white] (\y) ellipse[x radius=8mm,y radius=4mm];
\foreach \y in {A,B,C,D,E}
\draw[thick] (\y) ellipse[x radius=8mm,y radius=4mm];
\draw (A) node{0};
\draw (B) node{$S_2$};
\draw (D) node{$S_1$};
\draw (C) node{$mod\,\Lambda$};
\draw (E) node{$S_1,P_1$};

\begin{scope}[xshift=-1cm,yshift=1.5cm] 
\draw[fill,white](-.4,-.3)--(.4,-.3)--(.4,.3)--(-.4,.3)--(-.4,-.3);
\draw[thick] (-.4,-.3)--(.4,-.3)--(.4,.3)--(-.4,.3)--(-.4,-.3);
\end{scope}
\draw (AB) node{$S_2$};
\begin{scope}[xshift=-1cm,yshift=4.5cm] 
\draw[fill,white](-.4,-.3)--(.4,-.3)--(.4,.3)--(-.4,.3)--(-.4,-.3);
\draw[thick] (-.4,-.3)--(.4,-.3)--(.4,.3)--(-.4,.3)--(-.4,-.3);
\end{scope}
\draw (BC) node{$S_1$};
\begin{scope}[xshift=1.4cm,yshift=.9cm] 
\draw[fill,white](-.4,-.3)--(.4,-.3)--(.4,.3)--(-.4,.3)--(-.4,-.3);
\draw[thick] (-.4,-.3)--(.4,-.3)--(.4,.3)--(-.4,.3)--(-.4,-.3);
\end{scope}
\draw (AD) node{$S_1$};

\begin{scope}[xshift=2.8cm,yshift=3cm] 
\draw[fill,white](-.4,-.3)--(.4,-.3)--(.4,.3)--(-.4,.3)--(-.4,-.3);
\draw[thick] (-.4,-.3)--(.4,-.3)--(.4,.3)--(-.4,.3)--(-.4,-.3);
\end{scope}
\draw (DE) node{$P_1$};

\begin{scope}[xshift=1.4cm,yshift=5.1cm] 
\draw[fill,white](-.4,-.3)--(.4,-.3)--(.4,.3)--(-.4,.3)--(-.4,-.3);
\draw[thick] (-.4,-.3)--(.4,-.3)--(.4,.3)--(-.4,.3)--(-.4,-.3);
\end{scope}
\draw (EC) node{$S_2$};


\begin{scope}[xshift=8cm]
{\coordinate (A) at (0,0);
\coordinate (B) at (-2,3);
\coordinate (C) at (0,6);
\coordinate (D) at (2.8,1.8);
\coordinate (E) at (2.8,4.2);
\coordinate (AB) at (-1,1.5);
\coordinate (AB2) at (-2,0);
\coordinate (BC2) at (-2.1,6);
\coordinate (BC) at (-1,4.5);
\coordinate (AD) at (1.4,.9);
\coordinate (AD2) at (2.1,0);
\coordinate (DE) at (2.8,3);
\coordinate (EC) at (1.4,5.1);
\coordinate (EC2) at (2,6);
\coordinate (O) at (0,3);
\draw[very thick] (BC2)--(AD2);
\draw[very thick] (EC2)--(AB2);
\draw[very thick] (O)--(4,3);
\foreach \y in {A,B,C,D,E}
\draw[fill,white] (\y) ellipse[x radius=8mm,y radius=4mm];
\foreach \y in {A,B,C,D,E}
\draw[thick] (\y) ellipse[x radius=8mm,y radius=4mm];
\draw (A) node{0};
\draw (B) node{$S_2$};
\draw (D) node{$S_1$};
\draw (C) node{$mod\,\Lambda$};
\draw (E) node{$S_1,P_1$};

\begin{scope}[xshift=-1cm,yshift=1.5cm] 
\draw[fill,white](-.6,-.3)--(.6,-.3)--(.6,.3)--(-.6,.3)--(-.6,-.3);
\end{scope}
\draw (AB) node{$D(S_2)$};
\begin{scope}[xshift=-1cm,yshift=4.5cm] 
\draw[fill,white](-.6,-.3)--(.6,-.3)--(.6,.3)--(-.6,.3)--(-.6,-.3);
\end{scope}
\draw (BC) node{$D(S_1)$};
\begin{scope}[xshift=1.4cm,yshift=.9cm] 
\draw[fill,white](-.6,-.3)--(.6,-.3)--(.6,.3)--(-.6,.3)--(-.6,-.3);
\end{scope}
\draw (AD) node{$D(S_1)$};

\begin{scope}[xshift=2.8cm,yshift=3cm] 
\draw[fill,white](-.6,-.3)--(.6,-.3)--(.6,.3)--(-.6,.3)--(-.6,-.3);
\end{scope}
\draw (DE) node{$D(P_1)$};

\begin{scope}[xshift=1.4cm,yshift=5.1cm] 
\draw[fill,white](-.6,-.3)--(.6,-.3)--(.6,.3)--(-.6,.3)--(-.6,-.3);
\end{scope}
\draw (EC) node{$D(S_2)$};
}
\end{scope}

\end{tikzpicture}
\caption{Lattice of torsion classes for $A_2: 1\to2$ on left with torsion classes in ovals. Brick labels of edges in this Hasse diagram are in boxes. On the right is the semi-invariant picture with torsion class labels on chambers.}
\label{Fig: Torsion classes in A2}
\end{center}
\end{figure}
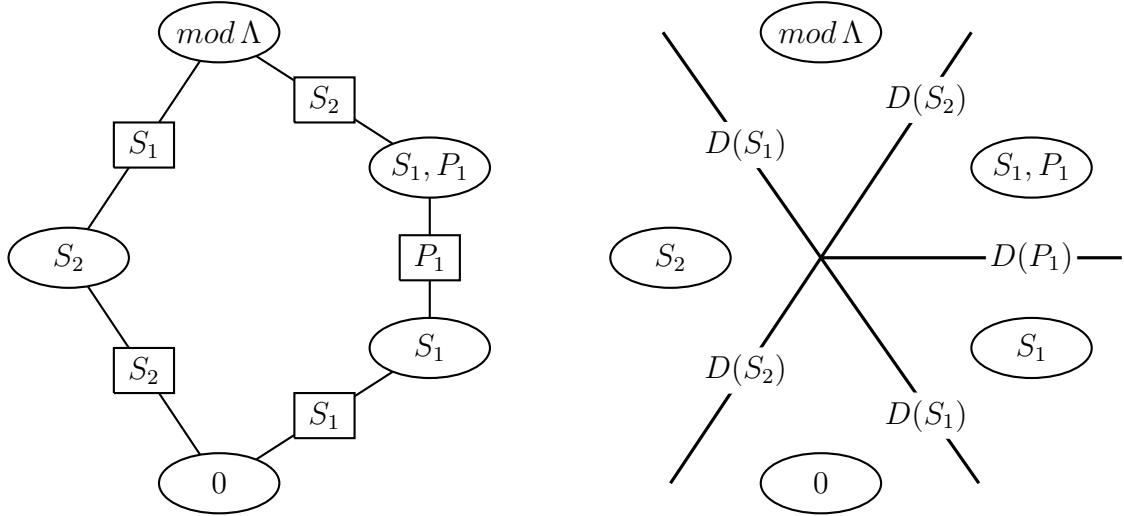
}

In Figure \ref{Fig: Stability diagram for A3} we show the semi-invariant picture for the $A_3$ quiver $1\to 2\to 3$ with walls $D(M)$ labeled with $M$ in a box and torsion classes in ovals in some of the chambers. {\color{blue}In the remainder of this paper we will examine
the signed exceptional sequences associated with the triangle labelled $T_3$ in Figure \ref{Fig: Edges of cube are cluster morphisms} below both algebraically and geometrically.} The cluster-tilting object at this simplex is $S_1\oplus S_3\oplus P_1$ since $a,b,c$ are the $g$-vectors of these modules. Figure \ref{Fig: Edges of cube are cluster morphisms} shows a blow up of this triangle with cube category showing the associated cluster morphisms. The cone of this triangle has cone point the origin in $\RR^3$. The solid 3-ball is the union of these cubes which are the cones over the 14 triangles in $S^2$ shown in Figure \ref{Fig: Stability diagram for A3}.

{
\begin{figure}[htbp]
\begin{center}
\begin{tikzpicture}[scale=1] 
\begin{scope}[xshift=6cm]
\coordinate(v0) at (0,-2.5);
\coordinate(v1) at (-1,-1);
\coordinate(v2) at (0,.75);
\coordinate(v3) at (1,2.5);
\coordinate(v4) at (1,-1);
\coordinate(v5) at (2.2,0);
\coordinate(v6) at (2.2,1.5);
\coordinate(v01) at (-.5,-1.75); 
\coordinate(v04) at (.5,-1.75); 
\coordinate(v12) at (-.5,-.12); 
\coordinate(v42) at (.5,-.12); 
\coordinate(v23) at (.5,1.62); 
\coordinate(v45) at (1.5,-.5); 
\coordinate(v63) at (1.5,2); 
\coordinate(v56) at (2.2,.75); 
\draw[thick] (v4)--(v0)--(v1)--(v2)--(v3)--(v6)--(v5)--(v4)--(v2);
\foreach \x/\y in {v0/T_0,v1/T_1,v2/T_2,v3/T_3,v4/T_4,v5/T_5,v6/T_6}
\draw[fill,white] (\x) ellipse[x radius=2mm,y radius=2mm];
\foreach \x/\y in {v0/T_0,v1/T_1,v2/T_2,v3/T_3,v4/T_4,v5/T_5,v6/T_6}
\draw (\x) node{$\y$};
\foreach \x/\y in {v01/T_0,v04/T_1,v12/T_2,v42/T_3,v23/T_4,v45/T_5,v63/T_6,v56/x}
\draw[fill,white] (\x) ellipse[x radius=2mm,y radius=2mm];
\foreach \x/\y in {v01/S_2,v04/S_1,v12/S_1,v42/S_2,v23/I_2,v45/I_2,v63/S_3,v56/P_1}
\draw (\x) node{\tiny$\y$};
\end{scope}
\begin{scope}[xshift=5.8cm,yshift=-2mm]
\draw[thick] (-.5,-1.75)rectangle (-.1,-1.35); 
\draw[thick] (.5,-1.75) rectangle (.9,-1.35); 
\draw[thick] (-.5,-.12)rectangle (-.1,.28); 
\draw[thick] (.5,-.12)rectangle (.9,.28); 
\draw[thick] (.5,1.62)rectangle (.9,2.02); 
\draw[thick] (1.5,-.5)rectangle (1.9,-.1); 
\draw[thick] (1.5,2)rectangle (1.9,2.4); 
\draw[thick] (2.2,.75)rectangle (2.6,1.15); 
\end{scope}
\coordinate (S1) at (-1,-4.3);
\coordinate (S2) at (3.5,2.5);
\coordinate (S3) at (-5.2,1.5);
\coordinate (P1) at (0,-2.2);
\coordinate (P2) at (2.2,1.5);
\coordinate (I2) at (1,-2.5);
\coordinate (T0) at (-4.5,-2);
\coordinate (T1) at (-3.7,1.9);
\coordinate (T2) at (-3.2,-.9);
\coordinate (T3) at (-2,-1);
\coordinate (T4) at (-2,-3);
\coordinate (T5) at (-.7,-2.55);
\coordinate (T6) at (-.7,-1.8);
\coordinate (A) at (-2.5,-1.9);
\coordinate (B) at (-2.4,1.4);
\coordinate (C) at (-.7,-1.1);
		\draw[ thick] (-2.25,1.3) circle [radius=3cm]; 
		\draw[ thick] (.75,1.3) circle [radius=3cm]; 
		\draw[ thick] (-.75,-1.3) circle [radius=3cm]; 
\begin{scope}
		\clip (-.75,-2) rectangle (4.25,4.5);
		\draw[ thick] (-.75,1.3) ellipse [x radius=3cm,y radius=2.6cm]; 
\end{scope}
\begin{scope}[rotate=60]
		\clip (0,-2.7) rectangle (-3.1,2.7);
		\draw[ thick] (0,0) ellipse [x radius=3cm,y radius=2.6cm]; 
\end{scope}
\begin{scope}
    \clip (-2.4,-3) rectangle (2,.1);
    \draw[ thick] (-.75,.43)  circle [radius=2.68cm]; 
\end{scope}
\coordinate (T7) at (-.7,0.3);
\foreach \x/\y in {A/a,B/b,C/c}
\draw (\x) node{$\y$};
\foreach \x/\y in {T0/T_0,T1/T_1,T2/T_2,T3/T_3,T4/T_4,T5/T_5,T6/T_6,T7/T_7}
\draw (\x) node{$\y$} (\x);
\foreach \x/\y in {T0/T_0,T1/T_1,T2/T_2,T3/T_3,T4/T_4,T5/T_5,T6/T_6,T7/T_7}
\draw[thick] (\x) ellipse[x radius=3.5mm, y radius=2.4mm];
{
\begin{scope}
\begin{scope}[xshift=-1cm,yshift=-4.3cm]
    \draw[fill,white] (-.5,-.2) rectangle (.5,.2);
    \draw (0,0)node{\small$D(S_1)$};
\end{scope}    
 \begin{scope}[xshift=3.5cm,yshift=2.5cm]
    \draw[fill,white] (-.3,-.2) rectangle (.3,.2);
    \draw (0,0)node{\small$D(S_2)$};
\end{scope}    
 \begin{scope}[xshift=-5.2cm,yshift=1.5cm]
    \draw[fill,white] (-.3,-.2) rectangle (.3,.2);
    \draw (0,0)node{\small$D(S_3)$};
\end{scope}    
\begin{scope}[xshift=0cm,yshift=-2.2cm]
    \draw[fill,white] (-.5,-.2) rectangle (.3,.2);
    \draw (0,0)node{\small$D(P_1)$};
\end{scope}     
\begin{scope}[xshift=2.2cm,yshift=1.5cm]
    \draw[fill,white] (-.3,-.2) rectangle (.3,.2);
    \draw (0,0)node{\small$D(P_2)$};
\end{scope}     
\begin{scope}[xshift=1cm,yshift=-2.5cm]
    \draw[fill,white] (-.3,-.2) rectangle (.3,.2);
    \draw (0,0)node{\small$D(I_2)$};
\end{scope}    
\end{scope}
}
\end{tikzpicture}
\caption{The semi-invariant picture for $A_3:1\to2\to3$ on the left. Portion of lattice of torsion classes on the right. At point $a$ we have the wide subcategory $\cW=\{S_3,P_1,I_2\}$ (and direct sums) since these are the walls at $a$. The chambers at $a$ give the 5 torsion classes $T_2,\cdots,T_6$ in the pentagon on the right.}
\label{Fig: Stability diagram for A3}
\end{center}
\end{figure}
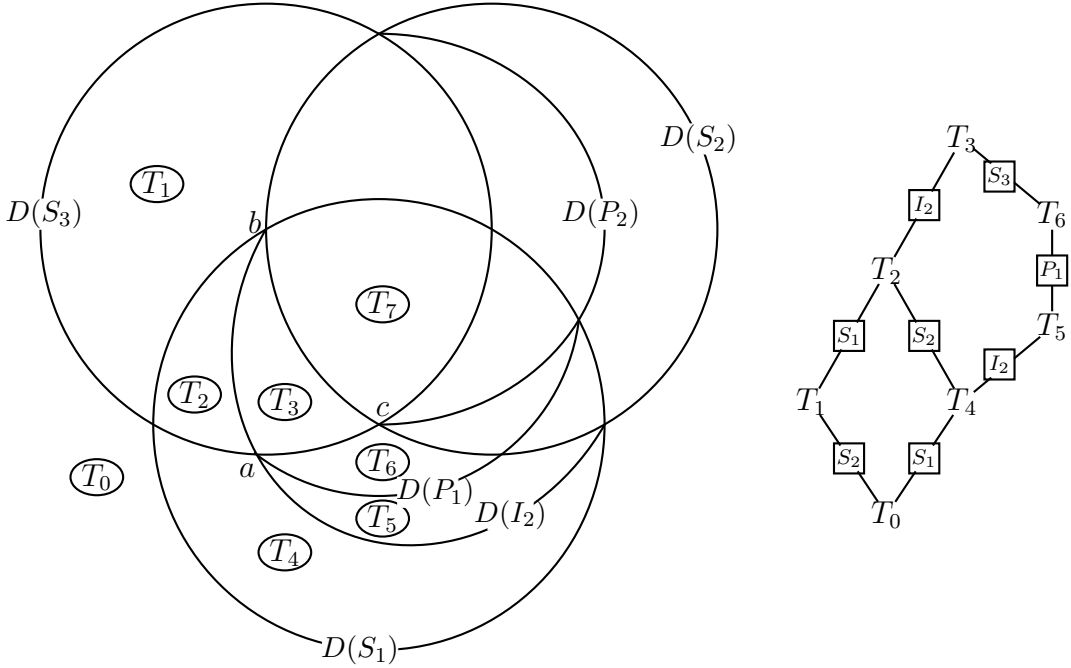
}

At the corners $a,b,c$ of the triangle there are wide subcategories $\cW,\cW',\cW''$ respectively. These are wide subcategories of rank $2$. We also have three wide subcategories of rank $1$ which lie on the edges of the triangle. We look at $\cW_0=\{S_2\}$. This is the rank 1 wide subcategory associated to the wall with brick label $S_2$. Since this wall has endpoints $b,c$ we have cluster morphisms $\cW'\to \cW_0$ and $\cW''\to \cW_0$. We will determine how these morphisms are labeled. The morphisms are defined up to sign:

$\cW'=\{S_1,S_2,I_2\}=S_3^\perp$. So, $mod\text-\Lambda\to \cW'$ is $S_3$ or $S_3[1]$.

$\cW_0=\{S_2\}=\cW'\cap I_2^\perp$. So, $\cW'\to \cW_0$ is $I_2$ or $I_2[1]$.

The morphism $\cW_0\to 0$ (at the center of the triangle) is $S_2[1]$ since the triangle labeled $T_3$ is on the negative side of the wall separating it from $T_7$ and $T_3\subset T_7$. The triangle is on the positive side of the other two walls. So, the other two morphisms to the center of the triangle are positive $I_2$ and $S_3$. This makes $\cW'\to \cW_0$ positive $I_2$.

$\cW''=\{S_3,S_2,P_2\}=P_1^\perp$. So, $mod\text-\Lambda\to \cW''$ is $P_1$ or $P_1[1]$.

$\cW_0=\{S_2\}=\cW''\cap S_3^\perp$. So $\cW''\to \cW_0$ is $S_3$ or $S_3[1]$. But it is parallel to the morphism $S_3$ to the center of the triangle. So, it must be positive.

To figure out the sign of the morphism $P_1$ we have to compare the two signed exceptional sequences coming from the sequences of morphisms:
\[
	mod\text-\Lambda\xrightarrow{S_3}\cW'\xrightarrow{I_2}\cW_0\xrightarrow{S_2[1]}0
\]
\[
	mod\text-\Lambda\xrightarrow{\pm P_1}\cW''\xrightarrow{S_3}\cW_0\xrightarrow{S_2[1]}0.
\]
These give signed exceptional sequences:
\[
	(S_2[1],I_2,S_3) \simeq (S_2[1],S_3,P_1 \,\text{or}\, P_1[1])
\]
which differ by the braid move $\sigma_2^{-1}$ given by $\sigma_2^{-1}(S_2[1],I_2,S_3) = (S_2[1],S_3,X)$ where, up to sign these are exceptional sequences (so $X=P_1$ or $P_1[1]$) and $\undim X$ is congruent to $\undim I_2=(1,1,0)$ modulo $\undim S_3=(0,0,1)$. So, $X=P_1$ with $\undim X=(1,1,1)$. So, the sign of $P_1$ is positive.

The six signed exceptional sequences labeling the 6 paths from the origin of $\RR^3$ to the center point of triangle labeled $T_3$ are all the factorizations of the cluster morphism $T_3:mod\text-\Lambda\to 0$. The first two, given above and in Figure \ref{Fig: Edges of cube are cluster morphisms} below are:
\begin{enumerate}
    \item $(S_2[1],I_2,S_3)$
    \item $(S_2[1],S_3,P_1)$
\end{enumerate}
The other four signed exceptional sequences are:
\begin{enumerate}
    \item[(3)] $(I_2,S_1,S_3)$, noting that, since $S_1$ is not relatively projective, it cannot be negative. Also $\undim S_1=\undim S_2[1]+\undim I_2$. So $(I_2,S_1,S_3)=\sigma_1^{-1}(S_2[1],I_2,S_3)$.
    \item[(4)] $(S_3,P_2[1],P_1)=\sigma_1^{-1}(S_2[1],S_3,P_1)$ since $P_2$ fits in the exceptional sequence and $\undim P_2[1]$ is congruent to $\undim S_2[1]$ modulo $\undim S_3$.
    \item[(5)] $(S_3,P_1,S_1)=\sigma_2^{-1}(S_3,P_2[1],P_1)$.
    \item[(6)] $(I_2,S_3,S_1)=\sigma_2(I_2,S_1,S_3)$.
\end{enumerate}
The last two are left to the reader to verify. Here the cluster corresponding to $T_3$ is $T_3=\{S_1,S_3,P_1\}$ since this is a cluster, but that was an accident since the torsion class $T_3$ is a cluster. Usually you take the support tilting object contained in the torsion class.

}

{ 
\begin{figure}[htbp]
\begin{center}
\begin{tikzpicture}[scale=1.7]
\coordinate (S1) at (-1,-4.3);
\coordinate (S2) at (3.5,2.5);
\coordinate (S3) at (-5.2,1.5);
\coordinate (P1) at (0,-2.2);
\coordinate (P2) at (2.2,1.5);
\coordinate (I2) at (1,-2.5);
\coordinate (T0) at (-4.5,-2);
\coordinate (T1) at (-3.7,1.9);
\coordinate (T2) at (-3.2,-.9);
\coordinate (T3) at (-2.2,-.6);
\coordinate (T3c) at (-2.25,-.63);
\coordinate (T3L) at (-2.3,-.63);
\coordinate (T3m) at (-2.2,-.65);
\coordinate (T4) at (-2,-3);
\coordinate (T5) at (-.7,-2.55);
\coordinate (T6) at (-.7,-1.8);
\coordinate (Ac) at (-2.4,-1.7);
\coordinate (A) at (-2.3,-1.9);
\coordinate (Am) at (-2.3,-1.6);

\coordinate (B) at (-2.3,1.35);

\coordinate (Bc) at (-2.24,1.3);

\coordinate (BcD)at (-1.22, 1);

\coordinate (D) at (-.2, .7);

\coordinate (Cc) at (-.73,-1.3);
\coordinate (CcD) at (-.45,-.3);

\coordinate (C) at (-.7,-1.66);

\draw[thick,->,blue] (D)--(-1.22, 1);
\draw[thick,blue] (Bc)--(-1.22, 1);
\draw[blue,fill] (Bc) circle[radius=.5mm];
\draw[blue] (Bc) node[below]{\quad\  \small$\cW'$};
\draw[fill,blue] (-1.8,-.3) circle[radius=.5mm];
\draw[blue] (-1.8,-.3) node[right]{\small$\cW_0=\{S_2\}$};
\draw[thick,blue,->](-1.8,-.3)--(T3);
\draw[fill,blue] (T3c) circle[radius=.5mm];
\draw[blue] (T3c) node[below]{\small$0$};
\draw[red,thick,->] (-1.5,-1.6)--(T3m);
\draw[thick,->,blue] (-2.24,1)--(-2.1, .4);
\draw[blue] (-2,.6) node{$I_2$};
\draw[blue] (-1.9,-.35) node[left]{\small$S_2[1]$};
\draw (Am) node[above]{\tiny$\cW$};
\draw[thick,red,->] (-.9,-1.2)--(-1.5,-.7);
\draw[red] (-1.5,-.7) node[right]{$S_3$};
\draw[red] (-1.85,-1.05)node[left]{$S_3$};

\draw[fill,red] (Cc) circle[radius=.5mm];
\draw[red] (Cc) node[above]{\small$\cW''$};

\draw[->] (-2.7,-.63)--(T3L);
\draw (-2.5,-.57)node[below]{\tiny$I_2$};
\draw[fill,blue] (D) circle[radius=.5mm];
\draw[blue] (D) node[right]{$mod\text-\Lambda$};
\draw (D) node[above]{$d$};

\draw[thick,->,red] (D)--(CcD);
\draw[thick,red] (Cc)--(CcD);
\draw[red] (CcD) node[right]{$P_1$};
\draw[->,thin] (D)--(Am);
\draw[thin] (Am)--(Ac);
\draw (A) node[left]{$g(S_1)=a$};
\draw (B) node[left]{$g(S_3)=b$};
\draw (C) node{$g(P_1)=c$};

\draw (-2.7,-.5) node[left]{\small$D(I_2)$};

\clip (-3.3,-2.3) rectangle (.2,2);
		\draw[very thick] (-2.25,1.3) circle [radius=3cm]; 
		\draw[very thick] (.75,1.3) circle [radius=3cm]; 
		\draw[very thick] (-.75,-1.3) circle [radius=3cm]; 
\begin{scope}
		\clip (-.75,-2) rectangle (4.25,4.5);
		\draw[very thick] (-.75,1.3) ellipse [x radius=3cm,y radius=2.6cm]; 
\end{scope}
\begin{scope}[rotate=60]
		\clip (0,-2.7) rectangle (-3.1,2.7);
		\draw[very thick] (0,0) ellipse [x radius=3cm,y radius=2.6cm]; 
\end{scope}
\begin{scope}
    \clip (-2.4,-3) rectangle (2,.1);
    \draw[very thick] (-.75,.43)  circle [radius=2.68cm]; 
\end{scope}
\draw[fill,white] (-1,1.7) circle[radius=3mm];
\draw (-1,1.7) node{\small$D(S_1)$};
\draw (-1.8,-1.7) node[above]{\small$D(S_3)$};
\draw (-2.1,0.2) node[right]{\small$D(S_2)$};

\draw[blue] (-1.2,1) node[below]{$S_3$};
%
\end{tikzpicture}
\caption{Vertices of the cube are labeled with wide subcategories. The edges are labeled with cluster morphisms. In blue we have the sequence of cluster morphisms of rank 1 where $\cW'=\{S_1,S_2,I_2\}$:
\[
	mod\text-\Lambda\xrightarrow{[S_3]}\cW'\xrightarrow{[I_2]}\cW_0\xrightarrow{[S_2[1]]} 0
\]
giving the signed exceptional sequence $(S_2[1],I_2,S_3)$. In red we have the following sequence of cluster morphisms where $\cW''=\{S_2,S_3,P_2\}$:
\[
	mod\text-\Lambda\xrightarrow{[P_1]}\cW''\xrightarrow{[S_3]}\cW_0\xrightarrow{[S_2[1]]} 0
\]
giving another signed exceptional sequence $(S_2[1],S_3,P_1)$. 
}
\label{Fig: Edges of cube are cluster morphisms}
\end{center}
\end{figure}
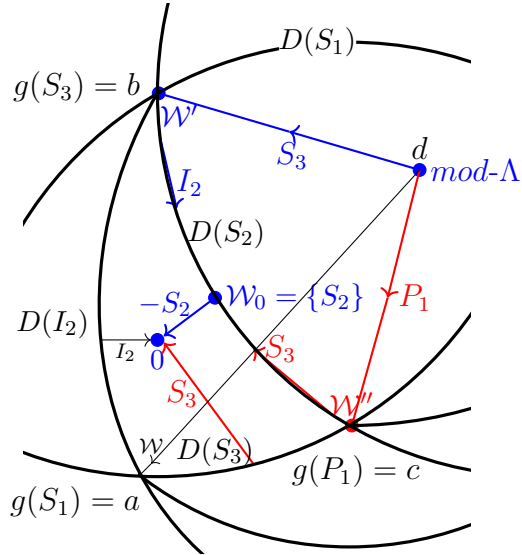
}

\section{Acknowledgments} The authors thank Kent Orr and Jerzy Weyman for their inspiration and collaboration in other works which lead to this work. We would not have discovered these concepts without them. The authors also thank the anonymous referee for their interest in this work and for their very thorough review which lead to many improvements in both exposition and content of the paper. Finally, the first author is grateful to the Simons Foundation for its generous support during the writing of this paper: Grant \#686616.

\end{document}